\documentclass[12pt]{amsart}
\usepackage{amsmath,palatino}
\usepackage{amssymb}
\usepackage{amsfonts}
\usepackage{latexsym}
\usepackage{amscd,enumerate}

\usepackage[all]{xy}

\input xy
\xyoption{all}

\numberwithin{equation}{section}

\newtheorem{lemma}{Lemma}[section]
\newtheorem{corollary}[lemma]{Corollary}
\newtheorem{theorem}[lemma]{Theorem}
\newtheorem{proposition}[lemma]{Proposition}
\newtheorem{remark}[lemma]{Remark}
\newtheorem{remarks}[lemma]{Remarks}
\newtheorem{definition}[lemma]{Definition}
\newtheorem{definitions}[lemma]{Definitions}

\newtheorem{problem}[lemma]{Problem}

\newtheorem*{remarknonum}{Remark}

\newcommand{\CC}{{\mathbb{C}}}
\newcommand{\MM}{{\mathbb{M}}}
\newcommand{\NN}{{\mathbb{N}}}
\newcommand{\QQ}{{\mathbb{Q}}}
\newcommand{\RR}{{\mathbb{R}}}
\newcommand{\ZZ}{{\mathbb{Z}}}

\newcommand{\C}{\mathcal{C}}
\newcommand{\F}{\mathcal{F}}
\newcommand{\M}{\mathcal{M}}
\renewcommand{\P}{\mathcal{P}}
\renewcommand{\O}{\mathcal{O}}

\newcommand\Zpl{{\ZZ^+}}
\newcommand{\bfz}{{\boldsymbol 0}}
\DeclareMathOperator\Aff{Aff}
 \DeclareMathOperator\LAff{LAff}
\newcommand\AffSu{\Aff(S_u)}
\newcommand\LAffsig{\LAff_\sigma(S_u)^{++}}
\newcommand\Wdsig{W^d_\sigma(S_u)}
\newcommand\extSu{\partial_eS_u}

\newcommand{\abar}{\overline{a}}
\newcommand{\bbar}{\overline{b}}
\newcommand{\cbar}{\overline{c}}
\newcommand{\ebar}{\overline{e}}
\newcommand{\fbar}{\overline{f}}
\newcommand{\pbar}{\overline{p}}
\newcommand{\qbar}{\overline{q}}
\newcommand{\sbar}{\overline{s}}

\newcommand{\ubar}{\overline{u}}
\newcommand{\wbar}{\overline{w}}
\newcommand{\xbar}{\overline{x}}
\newcommand{\ybar}{\overline{y}}
\newcommand{\zbar}{\overline{z}}
\newcommand{\alphabar}{\overline{\alpha}}
\newcommand{\betabar}{\overline{\beta}}
\newcommand{\onebar}{\overline{1}}

\DeclareMathOperator{\soc}{Soc}
 \DeclareMathOperator{\End}{End}

\newcommand{\leftmat}{\left(\begin{matrix}}
\newcommand{\rightmat}{\end{matrix}\right)}

\begin{document}
\author{G. Aranda Pino}
\address{Aranda and Siles: Departamento de \'Algebra, Geometr\'\i a y Topolog\'\i a, Universidad de M\'alaga, 29071 M{\'a}laga, Spain}
\email{gonzalo@agt.cie.uma.es, msiles@uma.es}
\author{K. R. Goodearl}
\address{Goodearl: Department of Mathematics, University of California at Santa Barbara, Santa Barbara, CA 93106, USA}
\email{goodearl@math.ucsb.edu}
\author{F. Perera}
\address{Perera: Departament de Matem\`atiques, Universitat Aut\`onoma de
Barcelona, 08193 Bellaterra, Barcelona, Spain}
\email{perera@mat.uab.cat}
\author{M. Siles Molina}

\thanks{G. A. P. and M. S. M. were supported in part by the Spanish MEC and Fondos FEDER through project MTM2007-60333, and jointly by the Junta de Andaluc\'{\i}a and Fondos FEDER through projects FQM-336 and FQM-2467; F. P. was supported in part by the DGI MEC-FEDER through Project MTM2005-00934 and the Comissionat per Universitats i Recerca de la Generalitat de Catalunya.}
\dedicatory{} \commby{} \keywords{} \subjclass{}

\date{}
\title[Non-simple purely infinite rings]{Non-simple purely infinite rings}

\maketitle

\begin{abstract}
In this paper we introduce the concept of purely infinite rings,
which in the simple case agrees with the already existing notion
of pure infiniteness. We establish various permanence properties
of this notion, with respect to passage to matrix rings, corners,
and behaviour under extensions, so being purely infinite is
preserved under Morita equivalence. We show that a wealth of
examples falls into this class, including important analogues of
constructions commonly found in operator algebras. In particular,
for any (s-)\linebreak unital $K$-algebra having enough nonzero
idempotents (for example, for a von Neumann regular algebra) its
tensor product over $K$ with many nonsimple Leavitt path algebras
is purely infinite.
\end{abstract}


\section*{Introduction}

The notion of pure infiniteness has proved key in the theory of
operator algebras since its conception in the early eighties by J.
Cuntz (see~\cite{Cunz}). This was done for simple algebras and
provided a huge list of examples. One of the milestones of the
theory became the classification of separable, nuclear, unital
purely infinite simple algebras by means of K-theoretic invariants
(due to Kirchberg~\cite{kir} and Phillips~\cite{phi}).

Far from being analytic, in the simple setting the notion of pure infinitess has a
strong algebraic flavour. Indeed, one of the various
characterizations states that a (unital) simple C$^*$-algebra $A$
is purely infinite if and only if $A\neq\mathbb{C}$ and for any
non-zero element $a$ in $A$, one has $xay=1$ for some elements $x$, $y$
in $A$. This led to P. Ara, the second named author and E. Pardo to
introduce a corresponding notion for rings (see~\cite{AGP}), as
follows. A simple ring $R$ is purely infinite if every non-zero
left (right) ideal contains an infinite non-zero idempotent (that
is, an idempotent that contains, isomorphically, a proper copy of
itself). As it happens, all (non-zero) idempotents in such rings are in fact properly infinite. They showed that a simple unital ring $R$ is purely infinite provided that $R$ is not a division ring and for any
non-zero element $a$ in $R$, one has $xay=1$ for suitable $x$ and
$y$ in $R$ (\cite[Theorem 1.6]{AGP}). This notion goes well beyond the
pure formalism to actually encompass a large number of algebras,
notably Leavitt's algebras of type $(1,n)$ (\cite{Leavitt}) and suitable von Neumann
regular extensions of those (as shown in~\cite{AGP}), as well as
Leavitt path algebras with suitable conditions on their defining
graphs (see~\cite{AA2}).

The operator algebraic notion just outlined above was extended
some years ago to the non-simple setting by Kirchberg and R\o rdam
(\cite{KR}) and has been studied intensively since then (see,
e.g.~\cite{KR2},~\cite{blankir},~\cite{blankir2},~\cite{KR3},~\cite{KP},~\cite{KNP}). Although there are
various possible formulations, the definition given in~\cite{KR}
appears to be more commonly used in possible extensions of the classification
programme to the non-simple case. The reader may wonder whether the extension consists simply of demanding that right or left ideals contain enough infinite idempotents. However, for technical reasons this turns out to be inappropriate. Instead, the approach to define purely infinite algebras resorts to the use of
the so-called Cuntz comparison for positive elements, which is
completely analytic and involves the so-called positive elements of the algebra. Roughly speaking, a C$^*$-algebra $A$ is said to be purely infinite if $A$ does not have abelian quotients and every pair of elements, with one contained in the closed two-sided ideal generated by the other, is suitably comparable. As with the simple case, all (non-zero) idempotents in such algebras are properly infinite.

This definition is mostly adequate when dealing with
algebras that might not have (non-trivial) idempotents. However,
\emph{if} all one-sided non-zero ideals in all quotients happen to
contain an infinite idempotent, this suffices to ensure pure
infiniteness, and takes into account the ideal structure of the
algebra.

In our aim to adapt this concept to the pure algebraic
setting, one of the major difficulties that we encounter is finding an
appropriate algebraic substitute for the analytic conditions. In
order to circumvent this, we introduce in Section 2 an analogue
for Cuntz comparison directed to general elements. We thus define a way to compare two elements in an
arbitrary ring which, in the case of idempotents, reduces to the usual (Murray-von Neumann) comparison. This allows us to define (general) properly infinite elements in a ring. As with idempotents, these are those that contain two orthogonal copies of themselves (see below for the precise definitions).

In Section 3 we introduce the concept of pure infiniteness for an
arbitrary, not necessarily unital, ring. There are at least two
different ways to do this that are both natural and accommodate an
expected generalisation from C$^*$-algebras. Within this class,
both ways of extending this concept are shown to be equivalent
(in~\cite[Theorem 4.16]{KR}), but they are different in the more
general framework. This is why our terminology needs to be
adapted, so we are bound to distinguish between \emph{properly
purely infinite} rings and \emph{purely infinite} rings. We prove
that every properly purely infinite ring is in fact purely
infinite. Thus the first class may be thought of as being purely
infinite in a strong sense, but we choose not to term them
strongly purely infinite in order to avoid confusion with the
corresponding notion for C$^*$-algebras (see~\cite{KR2}). We prove
that being purely infinite or properly purely infinite behaves
well when passing to quotients, ideals and in extensions. We also
prove that C$^*$-algebras that are purely infinite in our sense
are also purely infinite in the sense of~\cite{KR}.

Before analysing further permanence properties, we explore in Section 4 examples of purely infinite rings, and we already find interesting algebraic versions of analytic results. For example, if $A$ is any unital $K$-algebra over a field and $(B_i)_{i}$ is a sequence of unital $K$-algebras whose units are all properly infinite, then $A\otimes_K(\otimes_{i=1}^{\infty}B_i)$ is purely infinite (in fact, properly purely infinite) (see Theorem~\ref{inftensorexample}). We also deduce from Theorem~\ref{AtensorL} that $A\otimes_K L_K(1,\infty)$ is purely infinite for any von Neumann regular $K$-algebra $A$ (where $K$ is a field and $L_K(1,\infty)$ is the Leavitt algebra of type $(1,\infty)$).

The key question of whether corners and matrices over purely infinite rings are again purely infinite is addressed in Section 5. We prove that corners of purely infinite rings (resp. properly purely infinite rings) are again purely infinite (resp. properly purely infinite). Matrices turn out to be trickier, and we establish in Theorem~\ref{ExchangePiMatrices} that $M_n(R)$ is in fact properly purely infinite whenever $R$ is a purely infinite \emph{exchange} ring with local units, so that there is a good supply of idempotents. This result prompts the question of extending its validity to a larger class of rings, namely those whose monoids of isomorphism classes of finitely generated projective (right) modules have the Riesz refinement property. The typical example of this is that of the so-called Leavitt path algebras associated to graphs. We thus benefit from the results developed in order to completely characterize those Leavitt path algebras $L_K(E)$ associated to row-finite graphs that are purely infinite -- or, equivalently in this case, properly purely infinite (Theorem~\ref{LeavittPi}). As a consequence, given any unital $K$-algebra $A$ such that any nonzero right ideal in every quotient contains a nonzero idempotent, and given any row-finite graph $E$ for which $L_K(E)$ is purely infinite, the algebra $A\otimes_K L_K(E)$ is purely infinite.

\section{Preliminaries}

Throughout the paper $R$ will always denote a ring, which is not assumed to be unital. However, many of our results for nonunital rings require some replacement for the existence of a unit, such as one of the following conditions.

\begin{definitions} {\rm A ring $R$ is said to be \emph{s-unital} if for each $x\in R$, there exist
$u,v\in R$ such that $ux=xv=x$. When we say that an ideal $I$ of $R$ is \emph{s-unital}, we mean
that $I$ is s-unital when viewed as a ring in its own right, i.e., the elements $u$ and $v$ in the
definition must lie in $I$.

The defining property of an s-unital ring actually carries over to finite sets of elements, as the
following lemma of Ara shows. In particular, it follows that if $R$ is an s-unital ring, then all
the matrix rings $\MM_n(R)$ are s-unital.

The assertion that \emph{$R$ has local units} means that each
finite subset of $R$ is contained in a \emph{corner} of $R$, that
is, a subring of the form $eRe$ where $e$ is an idempotent of $R$.
(We will not use any of the more general types of corners
discussed in \cite{Lam}.) For example, every (von Neumann) regular
ring has local units \cite[Lemma 2]{FU}. Note that if $R$ has
local units, then $R$ is the directed union of its corners. We use
this term in the sense of, e.g.,~\cite{anhmarki} (rather
than~\cite{abrams}).

Finally, $R$ is said to be \emph{$\sigma$-unital} if there exists a countable sequence
$(u_1,u_2,\dots)$ of elements of $R$ such that
 \begin{enumerate}
 \item $u_{n+1}u_n= u_nu_{n+1}= u_n$ for all $n$.
 \item For any finite subset $F\subseteq R$, there is some $n\in\NN$ such that $u_nx= xu_n= x$
 for all $x\in F$.
 \end{enumerate}
 Observe that if $R$ is $\sigma$-unital and $R$ also has local units, then the sequence $(u_1,u_2,\dots)$
can be chosen so that all the $u_n$ are idempotents.
 }\end{definitions}

\begin{lemma} \label{sunital} {\rm \cite[Lemma 2.2]{arasunital}} Let $R$ be an s-unital ring. For
any finite subset $F\subseteq R$, there exists an element $u \in R$ such that $ux=xu=x$ for all
$x\in F$.  \end{lemma}

The situation is even better for s-unital exchange rings, by the following lemma of
Gonz\'alez-Barroso and Pardo. Recall from \cite[p. 412]{Ara} that a (possibly nonunital) ring $R$
is an \emph{exchange ring} if for each $x\in R$, there exist an idempotent $e\in R$ and elements
$r,s\in R$ such that $e= xr= x+s-xs$ (or the left-right symmetric version of this condition).

\begin{lemma} \label{sunitalexchange} {\rm \cite[Lemma 2.2]{GP}} Every s-unital exchange ring has
local units.  \end{lemma}

\begin{definitions} {\rm Recall that idempotents $e$ and $f$ in a ring $R$ are (Murray-von Neumann)
 \emph{equivalent}
(written $e\sim f$), provided there exist $x,y\in R$ such that $e=xy$ and $f=yx$ (after replacing
such $x$ and $y$ by $exf$ and $fye$, we may also assume that $x\in eRf$ and $y\in fRe$). This is
equivalent to demanding that $eR\cong fR$ as right $R$-modules. We say that $e$ and $f$ are
\emph{orthogonal} (which we denote by $e\perp f$) when $ef=fe=0$. In this situation, $e+f$ is also
idempotent and $(e+f)R=eR\oplus fR$. The \emph{orthogonal sum} of $e$ and $f$ is the idempotent
$e\oplus f = \leftmat e&0\\ 0&f \rightmat$ in $\MM_2(R)$. Orthogonal sums of larger (finite)
collections of idempotents are defined analogously. To deal with such idempotent matrices, the
definition of equivalence is extended in the natural way: idempotents $p\in \MM_m(R)$ and $q\in
\MM_n(R)$ are \emph{equivalent} if and only if there exist $x\in \MM_{m,n}(R)$ and $y\in
\MM_{n,m}(R)$ such that $p=xy$ and $q=yx$. For example, if $e$ and $f$ are orthogonal idempotents
in $R$, then $e+f \sim e\oplus f$.

We say that $e\leq f$ if $ef=fe=e$. We will also write $e<f$ when $e\leq f$ but $e\neq f$. We say
that $e$ is \emph{subequivalent} to $f$ (denoted $e\lesssim f$) when $e\sim g \leq f$ for some
idempotent $g\in R$. This holds if and only if there exist $x,y\in R$ such that $e=xfy$ (given
such $x$ and $y$, take $g=fyexf$). The idempotent $e$ is called \emph{infinite} if there exists an
idempotent $f$ such that $e\sim f<e$; equivalently, if there is a nonzero idempotent $g\in R$ such
that $e\sim e\oplus g$. If $e$ is not infinite, we say that $e$ is \emph{finite}; this holds if
and only if for any $x,y\in eRe$, we have $xy=e$ only if $yx=e$. Finally, $e$ is \emph{properly
infinite} provided $e\ne 0$ and $e\oplus e \lesssim e$.
 }\end{definitions}

Some of our results involve the monoid of equivalence classes of idempotent matrices over a ring.
We recall the construction here, along with some standard concepts associated with abelian
monoids.

\begin{definition} \label{VRdef} {\rm
Given a ring $R$ and an idempotent $e$ in a matrix ring $\MM_{\bullet}(R)$, write $[e]$ for the
(Murray-von Neumann) equivalence class of $e$. The set of these equivalence classes becomes an
abelian monoid, denoted $V(R)$, with respect to the addition operation given by $[e]+[f]= [e\oplus
f]$. (Alternatively, when $R$ is unital, $V(R)$ may be constructed as the monoid of isomorphism
classes of finitely generated projective right $R$-modules, with addition induced from direct
sum.) In case $R$ is a C$^*$-algebra, every idempotent matrix over $R$ is equivalent to a projection
matrix, and so the elements of $V(R)$ can be viewed as equivalence classes of projections.
 }\end{definition}

\begin{definitions} \label{monoidconcepts} {\rm Let $V$ be an abelian monoid. The \emph{algebraic
preordering} on $V$ is
the relation $\le$ defined as follows: $x\le y$ if and only if there exists $z\in V$ such that
$x+z=y$. (This relation is in general only reflexive and transitive, not necessarily
antisymmetric.) For idempotent matrices $e$ and $f$ over a ring $R$, we have $[e] \le [f]$ in
$V(R)$ if and only if $e\lesssim f$. Thus, for example, $e$ is properly infinite if and only if
$[e]\ne 0$ and $2[e] \le [e]$.

The monoid $V$ is \emph{conical} if $0$ is the only unit in $V$, that is, for $x,y\in V$ we have
$x+y=0$ only if $x=y=0$. For instance, $V(R)$ is conical.

An \emph{ideal} (or \emph{o-ideal}) in $V$ is any submonoid $I$ such that for all $x,y\in V$, we
have $x+y\in I$ only if $x,y\in I$ (equivalently, $I$ is hereditary with respect to the algebraic
preordering). Given an o-ideal $I$, there is a congruence $\equiv_I$ on $V$ defined as follows:
$x\equiv_I y$ if and only if there exist $a,b\in I$ such that $x+a= y+b$. We write $V/I$ for the
monoid $V/\equiv_I$, noting that such a quotient is always conical. As
with factor rings, we use overbars to denote congruence classes in quotient monoids. If $I$ is an
ideal of a ring $R$, then $V(I)$ is naturally isomorphic to a submonoid of $V(R)$. Moreover,
assuming $R$ is an exchange ring, $V(I)$ is an ideal of $V(R)$ with $V(R)/V(I) \cong V(R/I)$
\cite[Proposition 1.4]{AGOP}, and every ideal of $V(R)$ has the form $V(I)$ for some
(semiprimitive) ideal $I$ \cite[Teorema 4.1.7]{Par}.

An \emph{order-unit} for $V$ is an element $u\in V$ such that for each $x\in V$, there is some
$m\in\NN$ with $x\le mu$. This is the same as requiring that the ideal generated by $u$ equals
$V$. If $R$ is a unital ring, then $[1_R]$ is an order-unit in $V(R)$.

We say that $V$ is \emph{simple} (as an abelian monoid) if
 \begin{enumerate}
 \item There exist nonunits in $V$;
 \item the only ideals of $V$ are $V$ and the group of units of $V$.
 \end{enumerate}
In case $V$ is conical, it is simple if and only if it is nonzero and all nonzero elements are
order-units.
 }\end{definitions}

\section{Infinite and properly infinite elements}

Our basic definitions, like those in \cite{KR}, are in terms of equations involving $2\times2$
matrices over a ring. The following concepts will simplify manipulations with matrix equations.
While we use $\oplus$ in the same sense as \cite{KR}, our algebraic version of the relation
$\precsim$ differs from the Cuntz relation $\precsim$ (or $\lesssim$) for positive elements in a
C$^*$-algebra $A$. (It is closest to the relation $\underset\approx<$ defined in \cite[Section
1]{Cnz}, except that Cuntz's definition allows factors from the unitization of $A$.)

\begin{definitions} {\rm Let $R$ be a ring, and suppose $x$ and $y$ are square matrices over $R$, say
$x\in \MM_k(R)$ and $y\in \MM_n(R)$. We shall use $\oplus$ to denote block sums of matrices; thus,
 $$x\oplus y = \leftmat x&0\\ 0&y \rightmat \in \MM_{k+n}(R),$$
 and similarly for block sums of larger numbers of matrices. We define a relation $\precsim$ on
 matrices over $R$ by declaring that $x\precsim y$ if and only
 if there exist $\alpha\in \MM_{kn}(R)$ and $\beta\in \MM_{nk}(R)$ such that $x= \alpha y\beta$.

Observe that if $x$ and $y$ are idempotent matrices, then $x\precsim y$ if and only if $x\lesssim
y$. }\end{definitions}

\begin{lemma} \label{precsim} Let $R$ be a ring. For {\rm (iii)--(vi)}, assume that $R$ is
s-unital.
\begin{enumerate}[{\rm (i)}]
\item If $x$, $y$, $z$ are square matrices over $R$ and $x\precsim y\precsim z$, then $x\precsim z$.
\item If $x_1,y_1,\dots,x_n,y_n$ are square matrices over $R$ satisfying $x_i\precsim y_i$ for all
$i=1,\dots,n$, then $x_1\oplus \cdots\oplus x_n \precsim y_1\oplus \cdots\oplus y_n$.
\item If $x$ and $y$ are square matrices over $R$, then $x\precsim x$ and $x\oplus y\precsim y\oplus x$.
\item If $x$ is a square matrix over $R$, then $x\oplus\bfz\precsim x\oplus\bfz'$ for any square zero
matrices $\bfz$ and $\bfz'$. In particular, $x\precsim x\oplus \bfz\precsim x$ for any $\bfz$.
\item If $x,y\in \MM_k(R)$ for some $k$, then $xy,yx\precsim x$.
\item If $x_1,\dots,x_n \in \MM_k(R)$ for some $k$, then $x_1\pm x_2\pm \cdots\pm x_n \precsim x_1\oplus
\cdots\oplus x_n$.
\end{enumerate}
\end{lemma}

\begin{proof} (i) and (ii) are clear.

(iii) There exist square matrices $u$ and $v$ over $R$ such that $ux=xu=x$ and $vy=yv=y$. Then
$uxu=x$, whence
$x\precsim x$, and $\leftmat x&0\\ 0&y \rightmat= \leftmat 0&u\\ v&0 \rightmat \leftmat y&0\\
0&x \rightmat \leftmat 0&v\\ u&0 \rightmat$, showing that $x\oplus y\precsim y\oplus x$.

(iv) There is a square matrix $u$ over $R$ such that $ux=xu=x$. Inserting rectangular zero
matrices $\bfz_i$ of the appropriate sizes, we have
$$\leftmat x&\bfz_1\\
\bfz_2&\bfz \rightmat= \leftmat u&\bfz_3 \rightmat \leftmat x&\bfz_4\\
\bfz_5&\bfz' \rightmat \leftmat u\\ \bfz_6 \rightmat,$$
 whence $x\oplus\bfz\precsim x\oplus\bfz'$.

(v) There exists $u\in \MM_k(R)$ such that $ux=xu=x$. Since $xy=uxy$ and $yx=yxu$, we immediately
see that $xy,yx\precsim x$.

(vi) There exists $u \in \MM_k(R)$ such that $ux_i=x_iu=x_i$ for all $i$. Now
$$x_1\pm x_2\pm \cdots\pm x_n= \leftmat u &u &\cdots &u \rightmat \leftmat x_1 &0 &\cdots &0\\
0 &x_2 &&0\\ \vdots &&\ddots &\vdots\\ 0 &0 &\cdots &x_n \rightmat \leftmat u\\ \pm u\\ \vdots\\
\pm u \rightmat,$$
 whence $x_1\pm x_2\pm \cdots\pm x_n \precsim x_1\oplus \cdots\oplus x_n$.  \end{proof}

\begin{definition} \label{Kadef} {\rm
Let $R$ be a ring. For each element $a\in R$, we define
$$K(a)= \{x\in R \mid a\oplus x \precsim a \}. $$
 }\end{definition}

\begin{lemma} \label{equivKa} Let $R$ be a ring and $a,x\in R$. Then the following conditions are
equivalent:
\begin{enumerate}[{\rm (i)}]
\item $x\in K(a)$.
\item $a\oplus x\precsim a\oplus 0$.
\item There exist $\alpha_1,\alpha_2,\beta_1,\beta_2 \in R$ such that $\leftmat
a&0\\ 0&x \rightmat = \leftmat \alpha_1&0\\ \alpha_2&0 \rightmat \leftmat a&0\\ 0&0 \rightmat
\leftmat \beta_1&\beta_2\\ 0&0 \rightmat$.
\end{enumerate}
\end{lemma}

\begin{proof} (i) $\Rightarrow$ (iii). By assumption, there exist
 matrices $\leftmat \alpha_1\\ \alpha_2 \rightmat \in \MM_{21}(R)$ and $\leftmat \beta_1 &\beta_2
\rightmat \in \MM_{12}(R)$ such that $\leftmat  a & 0 \\ 0 & x \rightmat = \leftmat \alpha_1\\
\alpha_2 \rightmat a \leftmat \beta_1 &\beta_2 \rightmat$,
 from which it follows that
$$\leftmat
a&0\\ 0&x \rightmat = \leftmat \alpha_1&0\\ \alpha_2&0 \rightmat \leftmat a&0\\ 0&0 \rightmat
\leftmat \beta_1&\beta_2\\ 0&0 \rightmat.$$

(iii) $\Rightarrow$ (ii). This is clear.

(ii) $\Rightarrow$ (i). In the s-unital case, this is immediate from the fact that $a\oplus 0
\precsim a$ (Lemma \ref{precsim}(iv)). In general, (ii) gives us matrices $\alpha,\beta \in
\MM_2(R)$ such that $a\oplus x= \alpha (a\oplus 0) \beta$. If we write
 $\alpha= (\alpha_{ij})$ and $\beta= (\beta_{ij})$, then
$$\leftmat  a & 0 \\ 0 & x
\rightmat = \leftmat \alpha_{11}\\ \alpha_{21} \rightmat a \leftmat \beta_{11} &\beta_{12}
\rightmat,$$
 yielding $a\oplus x \precsim a$ and $x\in K(a)$.  \end{proof}

\begin{lemma} \label{Kisideal} If $R$ is an s-unital ring and $a\in R$, then
\begin{enumerate}[{\rm (i)}]
\item $K(a)$ is a two-sided ideal of $R$.
\item $K(a)\subseteq RaR$. In fact, each $x\in K(a)$ satisfies $x\precsim a$.
\end{enumerate}
\end{lemma}

\begin{proof}
(i) If $x\in K(a)$ and $r\in R$, then $a\oplus rx \precsim a\oplus x \precsim a$, whence $rx\in
K(a)$. Similarly, $xr\in K(a)$. If $x,y\in K(a)$, then
$$a\oplus (x\pm y) \precsim a\oplus x\oplus y \precsim a\oplus y \precsim a,$$
whence $x\pm y \in K(a)$.

(ii) If $x\in K(a)$, then $x\precsim 0\oplus x\precsim a\oplus x\precsim a$. Consequently, there
exist $\alpha,\beta \in R$ such that $x=\alpha a\beta \in RaR$.
\end{proof}

\begin{definitions} \label{infdef} {\rm We will say that an element $a$ in a ring $R$ is
\emph{infinite} if $K(a)\neq 0$, that is, if there exists a nonzero element $x\in R$ such that
$a\oplus x \precsim a$. We call an element $a\in R$ \emph{properly infinite} if $a\ne 0$ and $a\in
K(a)$, the latter condition being equivalent to $a\oplus a \precsim a$.
 Finally, we will say that $a\in R$ is
\emph{finite} if it is not infinite.}
\end{definitions}

\begin{remarks} \label{generalize} {\rm Note that, by Lemma \ref{Kisideal}(ii), when $R$ is an s-unital
ring, we get that $a\in R$ is properly infinite if and only if $K(a)=RaR$.

We observe that the concepts above agree with the classical ones when applied to an idempotent
$e\in R$, as follows. Here we do not need to assume s-unitality.

If $e$ is infinite in the usual sense, there is a nonzero idempotent $f\in R$ such that $e\oplus
f\sim e$, and so $e\oplus f \precsim e$, whence $e$ is infinite in the sense of Definition
\ref{infdef}.

Conversely, if $e$ is infinite in the sense of Definition \ref{infdef}, there exist
$x,\alpha_1,\alpha_2,\beta_1,\beta_2 \in R$ such that $x\ne 0$ and $\leftmat e&0\\ 0&x \rightmat =
\leftmat \alpha_1\\ \alpha_2 \rightmat e \leftmat \beta_1&\beta_2 \rightmat$, that is,
 $$\alpha_1 e\beta_1=e, \qquad \alpha_1 e\beta_2=\alpha_2 e\beta_1=0,
\qquad \alpha_2 e\beta_2= x.$$
 Since we can replace $\alpha_1$ and $\beta_1$ by $e\alpha_1e$ and $e\beta_1e$, there is no loss of
generality in assuming that $\alpha_1,\beta_1 \in eRe$. Now $\alpha_1\beta_1 =e$, and so $f=
\beta_1\alpha_1$ is an idempotent in $eRe$ with $f\sim e$. Since $f\beta_2= \beta_1\alpha_1
e\beta_2 =0$, we have $\alpha_2(e-f)\beta_2= \alpha_2e\beta_2= x\ne 0$, whence $f<e$. This shows
that $e$ is infinite in the usual sense.

Finally, assuming that $e\ne 0$, observe that $e$ is properly infinite in the sense of Definition
\ref{infdef} if and only if $e\oplus e \precsim e$, if and only if $e\oplus e \lesssim e$, i.e.,
if and only if $e$ is properly infinite in the usual sense. }
\end{remarks}

\begin{lemma}\label{finiteinquotient} If $R$ is an s-unital ring and $a\in R$, then $a+K(a)$ is
finite in $R/K(a)$.
\end{lemma}

\begin{remarknonum} {\rm In the following proof, and below, we use overbars to denote cosets in factor
rings.} \end{remarknonum}

\begin{proof}
Suppose that the coset $\abar \in R/K(a)$ is infinite. Then there exist $b\in R\setminus K(a)$ and
$\alpha_1,\alpha_2,\beta_1,\beta_2 \in R$ such that
 $$\leftmat \alphabar_1&0\\ \alphabar_2&0 \rightmat \leftmat \abar&0\\ 0&0
\rightmat \leftmat \betabar_1&\betabar_2\\ 0&0 \rightmat= \leftmat \abar&0\\ 0&\bbar \rightmat.$$
Now set
 $$x= \leftmat  a & 0 \\ 0 & b \rightmat - \leftmat  \alpha_1 & 0
\\ \alpha_2 & 0 \rightmat \leftmat  a & 0
\\ 0 & 0 \rightmat \leftmat  \beta_1 & \beta_2 \\ 0 & 0 \rightmat
 \in {\MM}_2(K(a)).$$
 Write $x=(x_{ij})$, with $x_{ij}\in K(a)$,
and choose $u \in R$ such that $ux_{ij}= x_{ij}u= x_{ij}$ for all $i$, $j$. We then have
 $$
 \leftmat \alpha_1 &u &u &0&0\\ \alpha_2 &0&0 &u &u \rightmat \leftmat a\\
 &x_{11}\\ &&x_{12}\\ &&&x_{21}\\ &&&&x_{22} \rightmat
 \leftmat \beta_1 &\beta_2\\ u&0\\ 0&u\\ u&0\\ 0&u \rightmat = \leftmat a&0\\
 0&b \rightmat,$$
 which shows that $a\oplus b \precsim a\oplus x_{11}\oplus x_{12}\oplus x_{21}\oplus x_{22}$.
 On the other hand, since $a\oplus x_{ij} \precsim a$ for all $i$, $j$, four successive applications
 of Lemma \ref{precsim} yield $a\oplus x_{11}\oplus x_{12}\oplus x_{21}\oplus x_{22} \precsim a$. But
 now $a\oplus b \precsim a$, contradicting the assumption that $b\notin K(a)$. Therefore $\abar$ is finite.
\end{proof}

\begin{corollary}\label{everyquotientfinite} For a nonzero element $a$ in an s-unital ring $R$,
 the following conditions are equivalent.
\begin{enumerate}[{\rm (i)}]
\item $a$ is properly infinite.
\item For any ideal $I$ of $R$, the coset $\abar= a+I$ is either zero or infinite.
\end{enumerate}
\end{corollary}

\begin{proof}
(i)$\Rightarrow$(ii). If $a\oplus a \precsim a$, then clearly $\abar\oplus \abar \precsim \abar$
in any quotient $R/I$, so that $\abar$ is either zero or properly infinite in $R/I$.

(ii)$\Rightarrow$(i). Suppose that $a$ is not properly infinite.
Since $a\neq 0$, we have $a\not\in K(a)$. Consider $I=K(a)$, which
is an ideal of $R$ by Lemma \ref{Kisideal}(i). Now $\abar$ is
nonzero and thus infinite in $R/K(a)$ by hypothesis, which
contradicts Lemma \ref{finiteinquotient}. Therefore $a$ is
properly infinite.
\end{proof}

\begin{proposition}\label{findtheproduct} Let $R$ be an s-unital ring. If $a\in R$ is properly
infinite and $b\in RaR$, then $b\precsim a$.
\end{proposition}

\begin{proof}
Write $b=\sum_{i=1}^n x_i a y_i$ for some $x_i,y_i \in R$. Each $x_iay_i \precsim a$, whence (by Conditions (vi) and (v) in Lemma~\ref{precsim})
$$b \precsim x_1ay_1\oplus x_2ay_2\oplus \cdots\oplus x_nay_n \precsim a\oplus a\oplus \cdots\oplus
a \precsim a,$$ because $a\oplus a \precsim a$. \end{proof}

\section{Purely infinite rings}

\begin{definitions} \label{defpi} {\rm We will say that a ring $R$ is \emph{purely infinite} if the
following two conditions
are satisfied:
\begin{enumerate}[{\rm (i)}]
\item No quotient of $R$ is a division ring.
\item Whenever $a\in R$ and $b\in RaR$, we have $b\precsim a$.
\end{enumerate}

We will say that $R$ is \emph{properly purely infinite} if every nonzero element of $R$ is
properly infinite.

These two concepts are closely related, as we will see below (cf. Lemmas
\ref{EveryPropInfImpliesPi}, \ref{MatricesAndPropInf} and Corollary \ref{exchangepinfstrong}).
}\end{definitions}

In \cite{AGP}, the authors gave a definition of ``purely infinite simple ring'' in the algebraic
setting by demanding that every nonzero right (or left) ideal contain an infinite idempotent. They
proved the following characterization.

\begin{theorem} \label{theoremagp} {\rm \cite[Theorem 1.6]{AGP}} Let $R$ be a unital
simple ring. Then $R$ is purely infinite if and only if
\begin{enumerate}[{\rm (i)}]
\item $R$ is not a division ring, and
\item $1_R\precsim a$ for every nonzero element $a\in R$.
\end{enumerate}
\end{theorem}

The concept of purely infinite simple ring was generalized to the
setting of rings with local units in \cite{AA2}, and of nonunital
(but $\sigma$-unital) rings in \cite{GP}. Concretely, the previous
characterization was generalized to the context of rings with
local units as follows.

\begin{proposition} \label{propositionaa2} {\rm \cite[Proposition 10]{AA2}} Let $R$ be a ring
with local units. Then $R$ is purely infinite simple if and only if
\begin{enumerate}[{\rm (i)}]
\item $R$ is not a division ring, and
\item $b\precsim a$ for all nonzero elements $a,b\in R$.
\end{enumerate}
\end{proposition}

Clearly then, Definition \ref{defpi} agrees with the previous
definitions in the case of simple rings with local units. Many
examples of purely infinite simple rings are known -- see,
e.g.,~\cite[Examples 1.3]{AGP},~\cite[Remark 2.7]{GP} and
also~\cite[Corollary 5.4]{AGGP}. For some classes of non-simple
(properly) purely infinite rings, see Sections \ref{examplesec}
and \ref{leavittsec}.

\begin{lemma} \label{EveryPropInfImpliesPi} Let $R$ be an s-unital ring.
\begin{enumerate}[{\rm (i)}]
\item If  $R$ is properly purely infinite, then it is purely infinite.
\item If $\MM_2(R)$ is purely infinite, then $R$ is properly purely infinite.
\end{enumerate}
\end{lemma}

\begin{proof}
(i) Suppose first that $R/I$ is a division ring for some ideal $I$ of $R$. Take a nonzero element
$\overline{a}$ of $R/I$. Then $a$ is a nonzero element in $R$ and by hypothesis it is properly
infinite. Find elements $\alpha_1,\alpha_2,\beta_1,\beta_2\in R$ such that
 $$\leftmat  a & 0 \\ 0 & a \rightmat =
\leftmat  \alpha_1 & 0 \\ \alpha_2 & 0 \rightmat \leftmat  a & 0 \\ 0 & 0 \rightmat \leftmat
\beta_1 & \beta_2 \\ 0 & 0 \rightmat.$$
 But then in $R/I$ we have that $$\leftmat  \abar & 0 \\
0 & \abar \rightmat= \leftmat  \alphabar_1 \abar \betabar_1 & \alphabar_1 \abar \betabar_2
\\ \alphabar_2 \abar \betabar_1 & \alphabar_2 \abar \betabar_2 \rightmat.$$
 Since $R/I$ is a division ring and $\abar\ne 0$, it follows that $\alphabar_1, \alphabar_2,
\betabar_1, \betabar_2$ are all nonzero. But then $\alphabar_1 \abar \betabar_2 =0$ implies
$\abar=0$, a contradiction.

This shows that no quotient of $R$ is a division ring. The other condition is obtained by invoking
Proposition \ref{findtheproduct}.

(ii) Given $a\in R$, there exists $u \in R$ such that $ua= au= a$. Hence,
 $$a\oplus a= \leftmat  u & 0 \\ 0 & 0 \rightmat \leftmat  a & 0 \\ 0 & 0 \rightmat
 \leftmat  u & 0 \\ 0 & 0 \rightmat + \leftmat  0 & 0 \\ u & 0
 \rightmat \leftmat  a & 0 \\ 0 & 0 \rightmat \leftmat  0 & u \\ 0 & 0 \rightmat \in \MM_2(R)
 (a\oplus 0) \MM_2(R).$$
Since $\MM_2(R)$ is assumed to be purely infinite, it follows that $a\oplus a\precsim a\oplus 0$,
and so $a\oplus a \precsim a$. Therefore $a$ is either zero or properly infinite.
\end{proof}

\begin{remark} \label{spiIMPLIESpi} {\rm If $R$ is a ring with zero multiplication, then
$R$ is trivially purely infinite, since $RaR=0$ for all $a\in R$. For similar trivial reasons, $R$
contains no infinite elements, and so $R$ is not properly purely infinite unless $R=0$. Thus,
proper pure infiniteness is, in general, stronger than pure infiniteness.}
\end{remark}

The example in Remark~\ref{spiIMPLIESpi} immediately suggests the problem below which has become quite elusive so far.

\begin{problem}
{\rm Find an s-unital ring $R$ which is purely infinite but not properly purely infinite.}
\end{problem}

Passage of (proper) pure infiniteness to ideals and quotients is given by the following result. We
will consider corners and matrix rings in Section \ref{matandcorner}.

Note that if $I$ is an s-unital ideal in a ring $R$, then any ideal $J$ of $I$ is also an ideal of
$R$. For if $x\in J$ and $r\in R$, then $x=ux$ for some $u\in I$, whence $rx= (ru)x \in Ix
\subseteq J$; similarly, $xr\in J$.

\begin{lemma} \label{pinftoidealquo} Let $I$ be an ideal of a ring $R$.
\begin{enumerate}[{\rm (i)}]
\item If $R$ is {\rm(}properly{\rm)} purely infinite, then so is $R/I$.
\item Now assume that $I$ is s-unital. If $R$ is {\rm(}properly{\rm)} purely infinite, then so is
$I$.
\end{enumerate}
\end{lemma}

\begin{proof} (i) It is clear that strong pure infiniteness passes from $R$ to $R/I$. Now assume
only that $R$ is purely infinite. Since any quotient of $R/I$ is also a quotient of $R$, no
quotient of $R/I$ is a division ring. Consider $a,b\in R$ such that $\bbar \in (R/I)\abar(R/I)$.
Then there is some $c\in RaR$ such that $\cbar= \bbar$. By hypothesis, $c=xay$ for some $x,y\in
R$, and therefore $\bbar= \cbar= \xbar\abar\ybar$.

(ii) Assume first that $R$ is properly purely infinite, and let $a\in I$ be nonzero. Then there
exist $\alpha_1,\alpha_2,\beta_1,\beta_2\in R$ such that
 $\leftmat  a & 0 \\ 0 & a \rightmat =
\leftmat  \alpha_1 \\ \alpha_2 \rightmat a \leftmat \beta_1 & \beta_2 \rightmat$. Since $I$ is
s-unital, we also have $a=ua=au$ for some $u \in I$. Then
 $$\leftmat  a & 0 \\ 0 & a \rightmat =
\leftmat  \alpha_1u \\ \alpha_2u \rightmat a \leftmat u\beta_1 & u\beta_2 \rightmat$$
 with $\alpha_1u,\alpha_2u,u\beta_1,u\beta_2\in I$. This proves that $I$ is properly purely
 infinite.

Now assume only that $R$ is purely infinite. If $a\in I$ and $b\in IaI$, then we at least have $b=
xay$ for some $x,y\in R$. Since also $a=ua=au$ for some $u \in I$, we have $b= (xu)a(uy)$ with
$xu,uy\in I$.

Suppose that $I$ has an ideal $J$ such that $I/J$ is a division ring. As noted above, $J$ is an
ideal of $R$. Since $R/J$ is purely infinite by (i), it suffices to find a contradiction working
in $R/J$. Thus, there is no loss of generality in assuming that $J=0$.

If $e$ is the identity element of $I$, then $I=eI=Ie$, and so $I=eR=Re$. It follows that
$er=ere=re$ for all $r\in R$, whence $e$ is a central idempotent of $R$. But then the annihilator
of $e$ in $R$ is an ideal $K$ such that $R= I\oplus K$, and $R/K \cong I$ is a division ring,
contradicting the assumption that $R$ is purely infinite. Therefore no quotient of $I$ is a
division ring.  \end{proof}

\begin{lemma} \label{propinftyinK} Let $R$ be an s-unital ring and $a,e\in R$ with $e=e^2\precsim a$.
\begin{enumerate}[{\rm (i)}]
\item If $e\sim f\leq e$ for an idempotent $f\in R$, then $e-f\in K(a)$.
\item If $e$ is properly infinite, then $e\in K(a)$.
\item If $e$ is infinite, then $a$ is infinite.
\end{enumerate}
\end{lemma}

\begin{proof}
(i) We first reduce to the case that $e\in aR$. By hypothesis, $e= \alpha a\beta$ for some
$\alpha,\beta \in R$, and there is no loss of generality in assuming that $e\alpha =\alpha$ and
$\beta e= \beta$. Set $e'= a\beta\alpha$ and $f'= a\beta f\alpha$. Then $e'$ and $f'$ are
idempotents, $e'\in aR$, and $f'\le e' \sim e$. Since $(a\beta) (f\alpha) =f'$ and
$(f\alpha)(a\beta) = f$, we have $f'\sim f$, and thus $e'\sim f'$. Similarly, $e'-f' \sim e-f$,
and so $e'-f'\in K(a)$ if and only if $e-f \in K(a)$. Thus, after replacing $e$ and $f$ by $e'$
and $f'$, there is no loss of generality in assuming that $e\in aR$.

Since $(e-f)\oplus e\sim (e-f)\oplus f\sim e$, we have $(e-f)\oplus e \precsim e$. Further,
$$a= ea+(a-ea) \precsim ea\oplus (a-ea) \precsim e\oplus (a-ea),$$
and hence
$$a\oplus (e-f)\precsim (e-f)\oplus a\precsim (e-f)\oplus e\oplus (a-ea)\precsim e\oplus (a-ea).$$

By hypothesis, $e=ax$ for some $x\in R$, and we note that $(a-ea)x=0$. Also, there exists $u \in
R$ such that $ua=au=a$. Set $w=u-xa$, and observe that $aw= a-ea$, whence $eaw=0$. Consequently,
$(a-ea)w= a-ea$. We now compute that
$$\leftmat e\\ u-e \rightmat a \leftmat x&w \rightmat=
\leftmat ea\\ a-ea \rightmat \leftmat x&w \rightmat = \leftmat e&0\\ 0&a-ea \rightmat,$$
 whence $e\oplus (a-ea) \precsim a$. Therefore $a\oplus (e-f) \precsim a$, which shows that $e-f\in K(a)$.

(ii). If $e$ is properly infinite, there are orthogonal idempotents $f,g\in R$ such that $f\oplus
g\leq e$ and $f\sim g\sim e$. Then $g\leq e$ and we can apply (i) to get that $e-g\in K(a)$. But
we also have $f\leq e-g$, so that $f=f(e-g)\in K(a)$, and consequently $e\in K(a)$ because $e\sim
f$.

(iii) If $e$ is infinite, there is an idempotent $f\in R$ with $e\sim f< e$. By (i), $e-f$ is a
nonzero element of $K(a)$, and therefore $a$ is infinite.
\end{proof}

\begin{proposition} \label{spiextn} Let $R$ be an s-unital ring and $I$ an s-unital ideal of $R$.
Assume that every nonzero ideal in every quotient of $I$ contains a nonzero idempotent. Then $R$
is properly purely infinite if and only if $I$ and $R/I$ are both properly purely infinite.
\end{proposition}

\begin{proof} Necessity follows from Lemma \ref{pinftoidealquo}. Conversely, assume that $I$ and
$R/I$ are properly purely infinite.

If $R$ is not properly purely infinite, there is a nonzero element $a\in R$ which is not properly
infinite. By Corollary \ref{everyquotientfinite}, $R$ has an ideal $J$ such that the coset
$\abar\in R/J$ is nonzero and finite. The ring $R'=R/J$ and its ideal $I'= (I+J)/J$ satisfy the
same hypotheses as $R$ and $I$, and so we may replace $R$, $I$, and $a$ by $R'$, $I'$, and
$\abar$. Thus, there is no loss of generality in assuming that $a$ is nonzero and finite.

Note that $a$ is not properly infinite, whence $a\notin I$. Since $R/I$ is properly purely
infinite, the coset $\abar \in R/I$ is properly infinite, and so there exist
$\alpha_1,\alpha_2,\beta_1,\beta_2\in R$ such that
$$\leftmat \abar &0\\ 0 &\abar \rightmat = \leftmat \alphabar_1\\ \alphabar_2 \rightmat \abar
\leftmat \betabar_1 &\betabar_2 \rightmat$$
 in $\MM_2(R/I)$. Observe that
$$\leftmat a &0\\ 0 &a \rightmat - \leftmat \alpha_1\\ \alpha_2 \rightmat a
\leftmat \beta_1 &\beta_2 \rightmat \in \MM_2(I\cap RaR),$$
 where we have $a\in RaR$ because $R$ is s-unital.
 The above difference cannot be zero, since $a$ is not properly infinite. Therefore $I\cap RaR \ne 0$.

By hypothesis, there exists a nonzero idempotent $e\in I\cap RaR$. Then $e= \sum_i r_ias_i$ for
some $r_i,s_i \in R$. Since we may replace the $r_i$ and $s_i$ by $er_i$ and $s_ie$, respectively,
there is no loss of generality in assuming that $r_i,s_i \in I$ for all $i$. Then, there is some
$u\in I$ such that $r_iu=r_i$ for all $i$. Now $ua\in I$ and $e= \sum_i r_i(ua)s_i\in IuaI$. Since
$I$ is properly purely infinite and hence purely infinite, $e\precsim ua$ in $I$.

Now $e\precsim a$ in $R$. Since $e$ is (properly) infinite, Lemma \ref{propinftyinK} implies that
$a$ is infinite, contradicting our assumptions. Therefore $R$ is indeed properly purely infinite.
\end{proof}

\begin{corollary} Let $R$ be an s-unital ring and $I$ a regular ideal of $R$. Then $R$ is properly purely
infinite if and only if $I$ and $R/I$ are both properly purely infinite. \end{corollary}

\begin{problem}
Does the above hold more generally? That is, for an s-unital ring $R$ and an s-unital ideal $I$ of $R$, is it the case that $R$ is properly purely infinite if and only if $I$ and $R/I$ both are?
\end{problem}

\begin{proposition} \label{IdealInfIdemImpliesPi} Let $R$ be  an s-unital ring. If
every nonzero right {\rm(}or left\/{\rm)} ideal in every nonzero quotient of $R$ contains an
infinite idempotent, then $R$ is properly purely infinite.
\end{proposition}

\begin{proof} Let $a$ be a nonzero element of $R$; we use Corollary \ref{everyquotientfinite} to see
that $a$ is properly infinite. Thus, let $I$ be an ideal of $R$ with $a\notin I$, and set $S
=R/I$. By hypothesis, the right ideal $\abar S$ contains an infinite idempotent $e$, so that there
is an idempotent $f\in S$ with $e\sim f<e$. Now, applying Lemma \ref{propinftyinK}(i) we have that
$0\neq e-f\in K(\abar)$, so that $\abar$ is infinite in $S$. Thus, Corollary
\ref{everyquotientfinite} shows $a$ is properly infinite, which yields in turn that $R$ is
properly purely infinite.
\end{proof}

The next proposition shows that, under suitable conditions, to prove that a ring $R$ is properly
purely infinite, we can relax the hypothesis that every nonzero element is properly infinite and
require only that every nonzero idempotent is properly infinite.

\begin{proposition} \label{SemiprimitiveExchangePi} Let $R$ be  an s-unital exchange ring. Suppose
that every ideal of $R$ is semiprimitive, and that all nonzero idempotents in $R$ are properly
infinite. Then $R$ is properly purely infinite.
\end{proposition}

\begin{proof}
We again use Corollary \ref{everyquotientfinite} to see that every nonzero element $a\in R$ is
properly infinite. Thus, it is enough to check that $\abar$ is infinite or zero in $R/I$, for an
arbitrary ideal $I$ of $R$.

Assume that $a\notin I$, and set $S =R/I$. By hypothesis, $S$ is a semiprimitive exchange ring.
This implies, in particular, that every nonzero right ideal $J$ of $S$ contains a nonzero
idempotent, as follows. Since $S$ is semiprimitive, $J \not\subseteq J(S)$, and so there exists an
element $x\in J$ which is not right quasiregular. Since $S$ is an exchange ring, there exist
$r,s\in S$ and an idempotent $e\in S$ such that $e=xr=s+x-xs$. Then $e\in J$, and $e\ne 0$ because
$x$ is not right quasiregular.

In view of the previous paragraph, the nonzero right ideal $\abar S$ contains a nonzero
idempotent, say $e$. This element can be lifted to a nonzero idempotent $f$ of $R$, which will be
properly infinite by hypothesis. Since the condition $f\oplus f\lesssim f$ passes to $e\oplus
e\lesssim e$, the idempotent $e$ is properly infinite in $S$, so Lemma \ref{propinftyinK}(ii)
applies to give us $0\neq e\in K(\abar)$. Therefore $\abar$ is infinite, as required.
\end{proof}

Since regular rings are s-unital semiprimitive exchange rings, and regularity passes to quotients,
we immediately obtain the following corollary. We shall prove later that a regular ring is purely
infinite if and only if it is properly purely infinite (see Corollary \ref{exchangepinfstrong}).

\begin{corollary} \label{RegularInftyIdempotentsPi} Let $R$ be a regular ring. Then $R$ is properly
purely infinite if and only if all nonzero idempotents in $R$ are properly infinite.
\end{corollary}

Next, we look at the relationship between the algebraic and analytic concepts of pure
infiniteness. First, recall the definition of pure infiniteness given by Kirchberg and R\o rdam in
\cite{KR}.

\begin{definition} \label{defpiKR} {\rm Let $A$ be a C$^*$-algebra. For $a,b\in A_+$, one defines $b \precsim
a$ (in the C$^*$ sense) to mean that there exists a sequence of elements $x_i\in A$ such that
$x_iax_i^* \rightarrow b$.

Now $A$ is \emph{purely infinite} in the sense of \cite[Definition 4.1]{KR} if the following two
conditions are satisfied:
\begin{enumerate}[{\rm (i)}]
\item There are no characters on $A$, that is, no nonzero homomorphisms $A\rightarrow \CC$.
\item For every $a,b\in A_+$ we have $b \precsim a$ if and only if $b\in \overline{AaA}$.
\end{enumerate} }
\end{definition}
Given a positive element $a$ in a C$^*$-algebra $A$ and $\epsilon>0$, write $(a-\epsilon)_+$ as the positive
part of $a-\epsilon\cdot 1$. In other words, $(a-\epsilon)_+=f(a)$, where $f\colon\mathbb{R}\to\mathbb{R}$ is
given by $f(t)=\max (t-\epsilon,0)$.

\begin{proposition} \label{PialgebraicPiKR} Let $A$ be a C$^*$-algebra. If $A$ is purely infinite in
the sense of Definition {\rm\ref{defpi}}, then it is purely infinite in the sense of Definition
{\rm\ref{defpiKR}}.
\end{proposition}

\begin{proof}
Observe that since $A$ has no quotients which are division rings,
it has no quotients isomorphic to $\CC$, and thus it has no characters.

Next, if $a,b\in A_+$ with $b\precsim a$, then obviously $b\in
\overline{AaA}$. Conversely, suppose that $b\in \overline{AaA}$. Given $\varepsilon
>0$, there exist $x_i\in A$ such that $\| b-\sum x_iax_i^*\| < \varepsilon$, and $\sum x_iax_i^*=
xay$ for some $x,y\in A$ because $A$ is purely infinite in the algebraic sense. Note that $xay\in
A_+$. Since $\| b-xay \| <\varepsilon$, we obtain $(b-\varepsilon)_+ \precsim xay$ by \cite[Lemma
2.5(ii)]{KR}. In addition, $xay\precsim a$ because of \cite[Proposition 2.4]{Ror}, so that
$(b-\varepsilon)_+\lesssim a$. Since $\varepsilon$ is arbitrary we conclude that $b\lesssim a$
\cite[Proposition 2.6]{KR}.
\end{proof}

\begin{remark} {\rm One might be tempted to look for a converse to Proposition \ref{PialgebraicPiKR},
at least for C$^*$-algebras with real rank zero, but we conjecture that no such converse holds.}
\end{remark}

\section{Tensor product and multiplier ring examples} \label{examplesec}

Various C$^*$-algebras obtained from tensor products or multiplier algebras are known to be purely
infinite. We develop some algebraic analogs in this section.

\begin{lemma} \label{elementtensorpropinf}
Let $R= A\otimes_K B$ where $A$ and $B$ are s-unital algebras over a field $K$, and let $a\in A$
and $b\in B$. If $a$ is nonzero and $b$ is properly infinite, then $a\otimes b$ is a properly
infinite element of $R$. \end{lemma}

\begin{proof} Since we are tensoring over a field, $a\otimes b \ne 0$.

By assumption, $\leftmat \alpha_1\\ \alpha_2 \rightmat b \leftmat \beta_1 &\beta_2 \rightmat =
\leftmat b&0\\ 0&b \rightmat$ for some $\alpha_i,\beta_j \in B$. Thus, $\alpha_i b\beta_j=
\delta_{ij}b$ for all $i$, $j$. As in the proof of Lemma \ref{EveryPropInfImpliesPi}(ii), there
exist $u_i,v_j\in \MM_2(A)$ such that
$$u_1(a\oplus 0)v_1+ u_2(a\oplus 0)v_2 = a\oplus a.$$
 Hence, we make the following computation in $\MM_2(A)\otimes_K B$:
 \begin{align*}(u_1\otimes \alpha_1+
u_2\otimes \alpha_2) \bigl( (a\oplus 0) \otimes b \bigr) ( v_1\otimes \beta_1+ v_2\otimes \beta_2)
&= \sum_{i,j=1}^2 u_i(a\oplus 0)v_j \otimes \alpha_i b\beta_j \\
 &= \sum_{i=1}^2 u_i(a\oplus 0)v_i
\otimes b= (a\otimes a) \otimes b. \end{align*}
 Thus, $(a\oplus a)\otimes b \precsim (a\oplus 0)\otimes b$. Under the usual identification of
 $\MM_2(A)\otimes_K B$ with $\MM_2(R)$, this last relation becomes $(a\otimes b)\oplus (a\otimes b)
 \precsim (a\otimes b) \oplus 0$. \end{proof}

Our first construction requires infinite tensor products of algebras. Recall that if
$B_1$, $B_2,\dots$ is an infinite sequence of unital algebras over a field $K$, the algebra
$\bigotimes_{i=1}^\infty B_i$ is defined to be the direct limit of the sequence $B_1\rightarrow
B_1\otimes_K B_2 \rightarrow B_1\otimes_K B_2\otimes_K B_3 \rightarrow \cdots$, where all the
connecting maps have the form $b\mapsto b\otimes 1$.

\begin{theorem} \label{inftensorexample}
Let $B= \bigotimes_{i=1}^\infty B_i$ and $R= A\otimes_K B$ where $A$ is an s-unital algebra over a
field $K$ and the $B_i$ are unital $K$-algebras. If the identity element of each $B_i$ is properly
infinite, then $R$ is properly purely infinite. \end{theorem}

\begin{proof} By construction, $R$ is the direct limit of a sequence of $K$-algebras $R_i$ and
injective connecting homomorphisms $\phi_i: R_i\rightarrow R_{i+1}$, where each $R_i= A\otimes_K
B_1\otimes_K \cdots\otimes_K B_i$ and $\phi_i(r)= r\otimes 1$ for $r\in R_i$. Since $A$ is
s-unital and the $B_i$ are unital, all the $R_i$ are s-unital. For any nonzero element $r\in R_i$,
Lemma \ref{elementtensorpropinf} implies that $\phi_i(r)$ is properly infinite in $R_{i+1}$.
Therefore all nonzero elements of $R$ are properly infinite. \end{proof}

Theorem \ref{inftensorexample} yields many properly purely infinite algebras without needing any
purely infinite simple algebras as ingredients. For example, $A$ could be an arbitrary nonzero
unital $K$-algebra and we could take each $B_i= \End_K(V_i)$ where the $V_i$ are infinite
dimensional vector spaces over $K$.

\begin{definition} {\rm Let $K$ be a field. The \emph{Leavitt algebra $L_K(1,\infty)$} (denoted
$U_\infty$ in papers such as \cite{AGP}) is the unital $K$-algebra with generators
$x_1,y_1,x_2,y_2,\dots$ and relations $x_iy_j= \delta_{ij}$ for all $i$, $j$. The following
notation is convenient for working with this algebra. Let $\F$ denote the set of all finite
sequences of positive integers, including the empty sequence, and set $x_\varnothing=
y_\varnothing =1\in L_K(1,\infty)$. For any nonempty sequence $I= (i_1,\dots,i_r) \in \F$, set
$x_I = x_{i_1} x_{i_2} \dots x_{i_r}$ and $y_I = y_{i_1} y_{i_2} \dots y_{i_r}$. Let $I^* =
(i_r,\dots,i_1)$ denote the reverse sequence to $I$, and note that $x_I y_{I^*} =1$. The products
$y_Jx_I$ for $I,J\in \F$ form a $K$-basis for $L_K(1,\infty)$.  }\end{definition}

The following fact is known, but we did not locate a reference in the literature.

\begin{lemma} \label{Lcentralsimple}
The algebra $L_K(1,\infty)$ is a central simple $K$-algebra, for any field $K$. \end{lemma}

\begin{proof} Simplicity of the algebra $L= L_K(1,\infty)$ is proved, for instance, in \cite[Theorem
4.3]{AGP}. Consider an element $c\in Z(L)$, and write $c= \sum_{I,J} \lambda_{I,J} y_Jx_I$ where
$I$ and $J$ run over $\F$, the $\lambda_{I,J} \in K$, and all but finitely many $\lambda_{I,J}
=0$. Choose an integer $n$ greater than all the entries in those $J$ for which some $\lambda_{I,J}
\ne 0$. Then $\lambda_{I,J} x_ny_J =0$ whenever $J\ne \varnothing$, and so $x_nc= \sum_I
\lambda_{I,\varnothing} x_nx_I$. On the other hand, $x_nc= cx_n= \sum_{I,J} \lambda_{I,J}
y_Jx_Ix_n$ where the $y_Jx_Ix_n$ are part of the standard basis for $L$. Hence, we must have
$\lambda_{I,J} =0$ whenever $J\ne \varnothing$. This allows us to rewrite $c$ in the form $c=
\sum_I \lambda_I x_I$ for suitable scalars $\lambda_I$. Now choose an integer $m$ greater than all
the entries in those $I$ for which $\lambda_I \ne 0$. Then $\lambda_I x_Iy_n =0$ whenever $I\ne
\varnothing$, and so $cy_n= \lambda_\varnothing$. On the other hand, $c_ny= yc_n= \sum_I \lambda_I
y_nx_I$, from which we conclude that $\lambda_I =0$ for all nonempty $I$. Therefore $c=
\lambda_\varnothing \in K$, proving that $Z(L) =K$.  \end{proof}

The following lemma is well known in the unital case (e.g., \cite[Theorem V.6.1]{Jac}); minor
modifications, which we leave to the reader, yield the s-unital case.

\begin{lemma} \label{idealtensorcentralsimple}
Let $R= A\otimes_K B$ where $A$ and $B$ are algebras over a field $K$. Assume that $A$ is
s-unital, that $B$ is unital and simple, and that $Z(B)=K$. Then every ideal of $R$ has the form
$I\otimes_K B$ for some ideal $I$ of $A$. \end{lemma}

Pardo has observed \cite{Ppriv} that the method used in the proof of \cite[Theorem 4.3]{AGP} can
be applied to show that $A\otimes_K L_K(1,\infty)$ is purely infinite simple for any unital simple
algebra $A$ over a field $K$. We thank him for permission to use this observation, which we extend
to the non-simple case in the following proof.

\begin{theorem} \label{AtensorL}
Let $R= A\otimes_K L_K(1,\infty)$ where $A$ is an s-unital algebra over a field $K$. Assume that
every nonzero right ideal in every quotient of $A$ contains a nonzero idempotent. Then $R$ is
properly purely infinite. \end{theorem}

\begin{proof} Set $L= L_K(1,\infty)$. By Proposition \ref{IdealInfIdemImpliesPi}, it suffices to
show that every nonzero right ideal in every nonzero quotient of $R$ contains an infinite
idempotent. In view of Lemmas \ref{Lcentralsimple} and \ref{idealtensorcentralsimple}, every
quotient of $R$ is isomorphic to an algebra of the form $(A/I)\otimes_K L$ where $I$ is an ideal
of $A$. Hence, after replacing $A$ by $A/I$, we just need to show that every nonzero right ideal
of $R$ contains an infinite idempotent.

We first claim that for any nonzero element $r\in R$, there is a nonzero element $a\in A$ such
that $a\otimes 1\precsim r$. Write $r= \sum_{I,J} a_{I,J} \otimes y_Jx_I$ where $I$ and $J$ run
over $\F$, the $a_{I,J} \in K$, and all but finitely many $a_{I,J} =0$. There is some $u\in A$
such that $ua_{I,J}
 = a_{I,J}u= a_{I,J}$ for all $I,J\in \F$. Choose $I'\in \F$ of
minimal size such that some $a_{I',J} \ne 0$, and then choose $J'\in \F$ of minimal size such that
$a_{I',J'} \ne 0$. For $I,J\in \F$, we have
 \begin{enumerate}[{\rm (i)}]
 \item If $a_{I,J}\ne 0$, then either $x_Iy_{I'^*}=0$ or $x_Iy_{I'^*}=x_{I''}$ for some $I''\in\F$,
 where $I''=\varnothing$
 only if $I=I'$.
 \item If $a_{I',J}\ne 0$, then either $x_{J'^*}y_J=0$ or $x_{J'^*}y_J=y_{J''}$ for some $J''\in\F$,
 where $J''=\varnothing$
 only if $J=J'$.
 \end{enumerate}
 Consequently,
$$(u\otimes x_{J'^*}) r (u\otimes y_{I'^*})= a\otimes 1+ \sum_{I,J\in\F} b_{I,J}\otimes y_Jx_I,$$
 where $a= a_{I',J'}\ne 0$ and $b_{\varnothing,\varnothing}
 =0$. Since the element $r'= (u\otimes x_{J'^*}) r (u\otimes y_{I'^*})$ satisfies $r'\precsim r$, we
 may replace $r$ by $r'$. Hence, there is no loss of generality in assuming that
 $a= a_{\varnothing,\varnothing}$ is nonzero. Now choose an integer $k$ greater than all the
 entries in those $I\in \F$ for which some $a_{I,J} \ne 0$, and greater than all the entries in
 those $J\in \F$ for which some $a_{I,J} \ne 0$. Then $a_{I,J} \otimes x_ky_Jx_Iy_k =0$
 whenever $I$ and $J$ are not both empty, and so
 $$(u\otimes x_k) r (u\otimes y_k)= a\otimes x_ky_k= a\otimes 1.$$
  Therefore $a\otimes 1 \precsim r$, as claimed.

 Let $H$ be a nonzero right ideal of $R$, choose a nonzero element $r\in H$, and let $a$ be a nonzero
 element of $A$ such that $a\otimes 1 \precsim r$.
 By hypothesis, there is a nonzero idempotent $e\in aA$, and $e\precsim a$ in $A$, whence $f= e\otimes
 1$ is a nonzero idempotent in $R$ such that $f\precsim a\otimes 1\precsim r$. Then $f= srt$ for
 some $s\in fR$ and $t\in Rf$, and so $g=rts$ is an idempotent in $rR$, equivalent to $f$. Lemma
 \ref{elementtensorpropinf} implies that $f$ is properly infinite, whence $g$ is properly
 infinite. Since $g\in rR \subseteq H$, the proof is complete. \end{proof}

\begin{corollary}
If $A$ is a regular algebra over a field $K$, then $A\otimes_K L_K(1,\infty)$ is a properly purely
infinite $K$-algebra.  \end{corollary}

Kirchberg and R\o rdam have proved that if $A$ is any C$^*$-algebra and $B$ any unital simple
separable purely infinite nuclear C$^*$-algebra, then $A\otimes B$ is purely infinite \cite[Theorem
4.5]{KR}. Theorem \ref{AtensorL} above corresponds to the case $B= \O_\infty$ of this result, but
it appears to be a difficult problem to find a general algebraic analog. The proof of
\cite[Theorem 4.5]{KR} uses the fact that $B\cong B\otimes (\bigotimes_{n=1}^\infty \O_\infty)$,
which is a consequence of classification theorems for C$^*$-algebras. Analogous results are not known
in the algebraic setting. In fact, it is not even known whether $L_K(1,\infty) \otimes_K
L_K(1,\infty) \cong L_K(1,\infty)$.

We now turn to multiplier rings for an additional source of examples. To supplement the following
definition, see, e.g., \cite{APe} for more detail.

\begin{definitions} {\rm
Let $R$ be a ring. A \emph{multiplier} (or \emph{double centralizer}) for $R$ is any pair of maps
$(f, g)\in \End(R_R)\times \End(_RR)$ which are \emph{compatible} in the sense that
$x\bigl(f(y)\bigr)= \bigl(g(x) \bigr)y$ for all $ x, y\in R$. Let $\M(R)$ denote the set of all
multipliers for $R$. It becomes a unital ring, called the \emph{multiplier ring of $R$}, in which
$$(f, g)+(f^\prime, g^\prime):= (f+f^\prime, g+g^\prime) \qquad\qquad \text{and} \qquad\qquad
(f, g) (f^\prime, g^\prime):= (ff^\prime, g^\prime g)$$
 for all $(f,g),(f',g') \in \M(R)$.

 There is a canonical homomorphism $\varphi : R \to \M(R)$ given by $x\mapsto(\lambda_x, \rho_x)$,
 where  $\lambda_x$ and $\rho_x$ denote the left and right multiplications by $x$, respectively,
 and the image $\varphi(R)$ is an ideal of $\M(R)$.
 The kernel of $\varphi$ consists of those $x\in R$ for which $xR= Rx= 0$. Thus, in particular,
 $\varphi$ is injective if $R$ is s-unital, in which case we identify $R$ with its image
 $\varphi(R) \subseteq
 \M(R)$.

 For example, if $S$ is a unital semiprime ring and $R= \MM_\infty(S)$ is the ring of all
 $\omega\times\omega$ matrices over $S$ with only finitely many nonzero entries, then $\M(R) =
 RCFM(S)$, the ring of all row- and column-finite $\omega\times\omega$ matrices over $S$
 \cite[Proposition 1.1]{APe}.
 }\end{definitions}

 \begin{theorem} \label{pinfmultiplierquo}
 Let $R$ be a $\sigma$-unital, nonunital, simple regular ring, and let $I$ be an ideal of
 $\M(R)$.
 \begin{enumerate}[{\rm (i)}]
 \item $\M(R)$ is an exchange ring but not a regular ring.
 \item The quotient $\M(R)/I$ is properly purely infinite if and only if all its nonzero
 idempotents are properly infinite.
 \end{enumerate}
 \end{theorem}

 \begin{proof} (i) \cite[Theorem 2]{OM} and \cite[Proposition 1.8]{APe}.

 (ii) By \cite[Theorem 2.5]{APe}, every ideal of $\M(R)$ is generated by idempotents.
 Consequently, every ideal of $\M(R)$ is semiprimitive, and therefore part (ii) follows from
 Proposition \ref{SemiprimitiveExchangePi}.  \end{proof}

The work of Ara and the third named author~\cite{APe} provides a number
of settings in which Theorem \ref{pinfmultiplierquo} can be
applied. The first of these is immediate:

\begin{corollary} \label{Mofpinfreg}
Let $R$ be a $\sigma$-unital, nonunital, purely infinite simple regular ring. Then $\M(R)$ is
properly purely infinite, and $R$ is the unique proper nonzero ideal of $\M(R)$.  \end{corollary}

\begin{proof} By \cite[Proposition 10]{AA2}, $R$ is also purely infinite in the sense of
\cite{AA2} and \cite[p. 3378]{APe}, i.e., every nonzero right ideal of $R$ contains an infinite
idempotent. Consequently, \cite[Proposition 2.13]{APe} shows that the identity map on $V(R)$
extends to an isomorphism
$$V(\M(R))\ {\overset\cong\longrightarrow}\ V(R) \sqcup \{\infty\}.$$
 In particular, $V(R)$ is
the unique proper nonzero ideal of $V(\M(R))$, and so it follows from \cite[Theorem 2.7]{APe} that
$R$ is the unique proper nonzero ideal of $\M(R)$.

The given description of $V(\M(R))$ implies that $2[e] =[e]$ for any idempotent $e\in \M(R)
\setminus R$, whence these idempotents are properly infinite. On the other hand, any nonzero
idempotent in $R$ is infinite (as already noted) and thus properly infinite, because $R$ is simple
(e.g., apply Corollary \ref{everyquotientfinite}). Therefore all nonzero idempotents in $\M(R)$
are properly infinite, and Theorem \ref{pinfmultiplierquo} implies that $\M(R)$ is properly purely
infinite.  \end{proof}

\begin{definitions} \label{defstatesetc} {\rm
Let $V$ be a nonzero abelian monoid which has an order-unit $u$. A \emph{state on $V$} is any
monoid homomorphism $s: V\rightarrow \RR_{\ge0}$ such that $s(u)=1$, and the \emph{state space of
$(V,u)$} is the set $S_u= S(V,u)$ of all such states. View $S_u$ as a subset of $\RR^V$, which is
a (locally convex, Hausdorff) linear topological space with the product topology, and observe that
$S_u$ is a compact convex subset of $\RR^V$. The \emph{extreme boundary} of $S_u$, that is, the
set of its extreme points, is denoted $\extSu$. We write $\AffSu$ for the set of all affine
continuous functions $S_u\rightarrow \RR$; this is a partially ordered real Banach space with
respect to the pointwise ordering and the supremum norm. There is a natural evaluation map
$\phi_u: V\rightarrow \AffSu$, such that $\phi_u(v)(f) = f(v)$ for $v\in V$ and $f\in \AffSu$. We
shall also need the set $\LAffsig$ consisting of those affine lower semicontinuous functions $S_u
\rightarrow (0,\infty]$ which are pointwise suprema of countable increasing sequences of strictly
positive functions from $\AffSu$.

An \emph{interval in $V$} is any nonempty hereditary upward directed subset $I$ of $V$. We say
that $I$ is \emph{countably generated} if it has a countable cofinal subset, and that $I$ is a
\emph{generating interval} if $I$ generates the monoid $V$. If $D$ is a countably generated
generating interval for $V$, and an order-unit $u\in V$ is given, we set $d= \sup \phi_u(D) \in
\LAffsig$ (the pointwise supremum of the functions in $\phi_u(D)$), and we define $\Wdsig$ to be
the following semigroup:
$$\{ f\in \LAffsig \mid f+g=nd \text{\ for some\ } g\in\LAffsig \text{\ and\ } n\in\NN \}.$$
 The disjoint union $V\sqcup \Wdsig$ can then be made into an abelian monoid using the given
 operations in $V$ and $\Wdsig$ together with the rule $x+f= \phi_u(x)+f$ for $x\in V$ and $f\in
 \Wdsig$.
}\end{definitions}

\begin{definition} {\rm A cancellative abelian monoid $V$ is \emph{strictly unperforated} if $nx<ny$
always implies $x<y$, for any $x,y\in V$ and $n\in\NN$.
 }\end{definition}

\begin{theorem} \label{APe211}  {\rm \cite[Theorem 2.11]{APe}}
Let $R$ be a $\sigma$-unital, nonunital, simple unit-regular ring. Assume that $\soc(R_R) =0$ and
that $V(R)$ is strictly unperforated.
\begin{enumerate}[{\rm (i)}]
 \item The set $D= \{ [e] \mid e=e^2 \in R \}$ is a countably generated generating interval in
 $V(R)$.
 \item Choose a nonzero element $u\in V(R)$, and define $S_u$, $\LAffsig$, $d$, $\Wdsig$ as in
 Definitions {\rm\ref{defstatesetc}}. Then the identity map on $V(R)$ extends to a monoid
 isomorphism
 $$\varphi: V(\M(R)) \rightarrow V(R) \sqcup \Wdsig$$
 such that $\varphi(x)= \sup \{ \phi_u(y) \mid y\in V(R) \text{\ and\ } y\le x \}$ for $x\in
 V(\M(R)) \setminus V(R)$.
\end{enumerate}
\end{theorem}

Notice that the space $S_u$ considered in Theorem~\ref{APe211} is
in fact a Choquet simplex. This is basically due to the fact that
$V(R)$ satisfies the Riesz refinement property (for a proof, see,
e.g.~\cite[Theorem 1.2]{Rocky}).

Kucerovsky and the third named author have used a C$^*$-algebraic
version of the above theorem to give sufficient conditions for
certain \emph{corona algebras} $\M(A)/A$ to be purely infinite
\cite[Lemma 3.1, Theorem 3.4]{KP} (see also~\cite{KNP}). Their argument carries over to
our algebraic context as follows.

\begin{theorem} \label{coronapinf}
Let $R$, $u$, $S_u$, $d$ be as in Theorem {\rm\ref{APe211}}. Assume that $\extSu$ is compact, that
the set $F_\infty= \extSu\cap d^{-1}(\{\infty\})$ consists of a finite number of isolated points
of $\extSu$, and that the restriction of $d$ to $F'_\infty= \extSu \setminus F_\infty$ is
continuous. Then $\M(R)/R$ is properly purely infinite, and it has at least $2^{|F_\infty|}$
distinct ideals.
\end{theorem}

\begin{proof} Since $\M(R)$ is an exchange ring (Theorem \ref{pinfmultiplierquo}(i)), idempotents
lift from $\M(R)/R$ to $\M(R)$. Thus, to verify that the nonzero idempotents in $\M(R)/R$ are
properly infinite, we need only consider cosets $\ebar= e+R$ for idempotents $e\in \M(R) \setminus
R$.

The isomorphism $\varphi$ in Theorem \ref{APe211} sends $[e]$ to a function $f\in \Wdsig$. Then
$f+g= md$ for some $g\in \Wdsig$ and $m\in\NN$. Since $d$ is finite and continuous on $F'_\infty$
and $g$ is lower semicontinuous, we see that $f$ is upper semicontinuous on $F'_\infty$ and hence
continuous there. Moreover, $F'_\infty$ is compact (because $F_\infty$ is open), and so $f$ is
bounded on $F'_\infty$. Thus, we may choose $k\in\NN$ such that $f(s) <k$ for all $s\in
F'_\infty$.

Now set $X= \{ s\in F_\infty \mid f(s) =\infty \}$. By increasing $k$ if necessary, we may assume
that $f(s)<k$ for all $s\in F_\infty \setminus X$. Then define $h: \extSu \rightarrow (0,\infty]$
so that
$$h(s)= \begin{cases} k-f(s) &\qquad (s\in \extSu \setminus X) \\ \infty &\qquad (s\in X). \end{cases}$$
 Since $h|_{F'_\infty}$ is continuous and $F_\infty$ is finite, we see that $h$ is lower
 semicontinuous; in fact, $h$ is the pointwise supremum of a countable sequence of continuous
 strictly positive functions. Compactness of $\extSu$ implies that the restriction map $\AffSu
\rightarrow C(\extSu,\RR)$ is an isomorphism of partially ordered
Banach spaces (e.g., \cite[Corollary 11.20]{poagi}), from which it
follows that $h$ extends uniquely to a map $h\in \LAffsig$.

Observe that $ku+f= 2f+h$ on $\extSu$, whence $ku+f=2f+h$ in $\LAffsig$. Since $ku+f \in \Wdsig$,
it follows that $h\in \Wdsig$. Hence, $h= \varphi([q])$ for some idempotent matrix $q$ over
$\M(R)$. Now $ku+[e]= 2[e]+[q]$ in $V(\M(R))$, whence $p\oplus \cdots\oplus p \oplus e \sim
e\oplus e\oplus q$ for some idempotent matrix $p$ over $R$. Passing to idempotent matrices over
$\M(R)/R$ yields $\ebar\sim \ebar\oplus \ebar\oplus \qbar$, and therefore $\ebar$ is properly
infinite, as desired.

We have now shown that all nonzero idempotents in $\M(R)/R$ are properly infinite. Therefore
Theorem \ref{pinfmultiplierquo} implies that $\M(R)/R$ is properly purely infinite.

Suppose that $F_\infty$ consists of distinct points $s_1,\dots,s_n$. For $i=1,\dots,n$, define
$h_i: \extSu \rightarrow (0,\infty]$ so that
$$h_i(s)= \begin{cases} \infty &\qquad (s=s_i)\\ 1 &\qquad (s\ne s_i). \end{cases}$$
As with $h$ above, we see that $h_i$ is lower semicontinuous and that it extends uniquely to a map
in $\LAffsig$. Since $d$ is positive and continuous on $S_u$, it is bounded below, and so there is
some $m\in\NN$ such that $md(s)>1$ for all $s\in S_u$. Hence, we can construct a function $g\in
\LAffsig$ such that
$$g(s)= \begin{cases} \infty &\qquad (s\in F_\infty)\\ md(s)-1 &\qquad (s\in F'_\infty). \end{cases}$$
 Then $h_i+g= md$ for all $i$, which shows that the $h_i \in \Wdsig$.

Now there exist idempotents $e_i\in \M(R) \setminus R$ such that $\varphi([e_i]) =h_i$ for
$i=1,\dots,n$. By what we have already proved, the idempotents $\ebar_i \in \M(R)/R$ are properly
infinite. Observe that if $h_i \le h_{j_1}+ \cdots+ h_{j_r}$ for some $i,j_1,\dots,j_r \in
\{1,\dots,n\}$, then $h_{j_1}(s_i)+ \cdots+ h_{j_r}(s_i) =\infty$ and so some $j_l= i$. It follows
that if $i,j_1,\dots,j_r$ are distinct elements of $\{1,\dots,n\}$, then $\ebar_i$ cannot belong
to the ideal of $\M(R)/R$ generated by $\ebar_{j_1}, \dots, \ebar_{j_r}$. Therefore the ideals of
$\M(R)/R$ generated by the different subsets of $\{\ebar_1,\dots,\ebar_n\}$ are all distinct.
\end{proof}

Recall that an \emph{ultramatricial algebra} over a field $K$ is any direct limit of a countable
sequence of finite direct products of matrix algebras $\MM_\bullet(K)$. Such an algebra $R$ is
always $\sigma$-unital and unit-regular, and $V(R)$ is strictly unperforated.

\begin{corollary} \label{coronaultramat}
Let $R$ be a nonunital simple ultramatricial algebra over a field $K$, and assume that $\soc(R_R)
=0$. Choose a nonzero element $u\in V(R)$, and assume that the state space $S(V(R),u)$ has only
finitely many extreme points.
\begin{enumerate}[{\rm (i)}]
 \item  $\M(R)/R$ is properly purely infinite.
 \item If there are distinct $s_1,\dots,s_n \in \partial_eS(V(R),u)$ such that $\sup \{ s_i([e])
 \mid e=e^2\in R \} =\infty$ for all $i=1,\dots,n$, then $\M(R)/R$ has at
least $2^n$ distinct ideals.
\end{enumerate}
\end{corollary}

\begin{corollary} \label{coronastableultramat}
Let $S$ be a unital simple ultramatricial algebra over a field $K$. Assume that $\soc(S_S) =0$ and
that the state space $S(V(S),[1_S])$ has only finitely many, say $n$, extreme points. Then
$RCFM(S)/\MM_\infty(S)$ is properly purely infinite, and it has at least $2^n$ distinct ideals.
\end{corollary}

\begin{proof} Set $R= \MM_\infty(S)$, and recall from \cite[Proposition 1.1]{APe} that $\M(R)=
RCFM(S)$. Observe that $R$ is a nonunital simple ultramatricial $K$-algebra with $\soc(R_R) =0$,
and that the natural embedding of $S$ into the upper left corner subalgebra of $R$ induces an
isomorphism $V(S)\rightarrow V(R)$. It follows that $V(R)= \{ [e] \mid e=e^2\in R \}$, and
consequently
$$\sup \{ s([e]) \mid e=e^2\in R \} =\infty$$
 for all $s\in S(V(R),[1_S])$. Therefore
Corollary \ref{coronaultramat} applies.
\end{proof}

Examples of the situation in Corollary \ref{coronastableultramat} for which $S(V(S),[1_S])$ has
arbitrarily many extreme points are easily obtained. For instance, let $n\in\NN$, and give the
abelian group $G= \QQ^n$ the \emph{strict ordering}, so that
$$G^+= \{0\} \cup \{ (x_1,\dots,x_n) \in G \mid x_i>0 \text{\ for all\ } i=1,\dots,n \}.$$
 Then $G$ is a countable simple dimension group with order-unit $u= (1,\dots,1)$. By, e.g.,
 \cite[Theorem 15.24 and Corollary 15.21]{vnrr}, there exists a unital simple ultramatricial $K$-algebra $S$ such that $(K_0(S),[1_S]) \cong
 (G,u)$ (as partially ordered abelian groups with distinguished order-units). In particular, $(V(S),[1_S])
 \cong (G^+,u)$. Since $G^+$ has no minimal positive elements, $S$ has no minimal idempotents, and
 so $\soc(S_S) =0$. It is easily checked that $S(G^+,u)$ consists of convex combinations of the
 $n$ canonical projection maps $\pi_i: G^+\rightarrow \QQ^+$, and thus $\partial_e S(G^+,u) = \{\pi_1,
\dots, \pi_n\}$.

\section{Matrices and corners of purely infinite rings} \label{matandcorner}

In the current section we study the passage of (proper) pure infiniteness to corners and matrices. As we prove below, corners inherit (proper) pure infiniteness in full generality, and in fact local rings at elements do for s-unital rings. Our arguments for the analysis of matrix rings require a certain abundance of idempotents, that the (large) class of exchange ring has. Hence, we prove that matrices over s-unital purely infinite exchange rings are properly purely infinite, from which we deduce that pure infiniteness is a Morita invariant property (for s-unital exchange rings).

\begin{lemma}\label{PrimePiNoDivCorners} Let $R$ be a purely infinite prime ring. Then no corner of $R$ is a
division ring.
\end{lemma}

\begin{proof}
Suppose that we have an idempotent $e \in R$ such that $eRe$ is a division ring. Since $R$ is
prime, $eR$ is a simple right $R$-module.

If $eR=R$, then $(R(1-e))^2=0$ and so $R(1-e)=0$ because $R$ is prime. (Here we are writing
$R(1-e)$ for the left ideal $\{r-re \mid r\in R\}$.) But then $R=eRe$ and $R$ is a division ring,
contradicting the hypothesis that $R$ is purely infinite. Thus, $eR\ne R$ and so $(1-e)R \ne 0$.
Now $(1-e)ReR\ne 0$ because $R$ is prime, and hence there exists a nonzero element $a\in (1-e)Re$.
Note that $aR$ is a nonzero homomorphic image of $eR$, whence $aR$ is a simple right $R$-module.
Since $R$ is prime, $aR= gR$ for some idempotent $g$, and $eg=0$ because $ea=0$. Observe that
$g-ge$ is an idempotent which generates $gR$, so we can replace $g$ by $g-ge$. Hence, there is no
loss of generality in assuming that $e\perp g$.

Now $f=e+g$ is an idempotent such that $fR= eR\oplus aR$, and $f\in ReR$ because $gR= aR \subseteq
ReR$. Since $R$ is purely infinite, $f=xey$ for some $x,y\in R$. But then $fR$ is a homomorphic
image of $eR$, implying that $fR$ is simple or zero, which is impossible in light of $fR= eR\oplus
aR$. This contradiction establishes the lemma.
\end{proof}

\begin{proposition} \label{PiCorners} Let $e$ be an idempotent in a ring $R$. If $R$ is
{\rm(}properly{\rm)} purely infinite, then so is $eRe$.
\end{proposition}

\begin{proof}
Assume first that $R$ is properly purely infinite. Any nonzero element $a\in R$ is properly
infinite in $R$, and so
 $\leftmat  a & 0 \\ 0 & a \rightmat =
\leftmat  \alpha_1 \\ \alpha_2 \rightmat a \leftmat \beta_1 & \beta_2 \rightmat$ for some
$\alpha_1,\alpha_2,\beta_1,\beta_2\in R$. Then
 $$\leftmat  a & 0 \\ 0 & a \rightmat =
\leftmat  e\alpha_1e \\ e\alpha_2e \rightmat a \leftmat e\beta_1e & e\beta_2e \rightmat,$$
 which shows that $a$ is properly infinite in $eRe$. Therefore $eRe$ is properly purely infinite
 in this case.

Now assume only that $R$ is purely infinite. Suppose first that
$I$ is an ideal of $eRe$ such that $eRe/I$ is a division ring. In
this case $I$ is a maximal ideal of $eRe$. Moreover, $e\notin
(eRe)I(eRe)= eRIRe$, and so $e\notin RIR$. Consequently, $\ebar$
is a nonzero idempotent in $R/RIR$, and in particular, $\ebar$
cannot be in the Jacobson radical of $R/RIR$. Hence, there exists
a (left) primitive ideal $P$ of $R$ such that $e\notin P$ and
$RIR\subseteq P$. Now $I\subseteq P\cap eRe\subsetneq eRe$, and by
maximality of $I$ in $eRe$ we have $I=P\cap eRe$. This entails
$eRe/I=eRe/(P\cap eRe) \cong \ebar(R/P) \ebar$. But this means
that the purely infinite prime ring $R/P$ has a corner which is a
division ring, contradicting Lemma \ref{PrimePiNoDivCorners}.
Therefore no quotient of $eRe$ is a division ring.

The other condition is easy. Suppose that $a\in eRe$ and $b\in (eRe)a(eRe)\subseteq RaR$. Since
$R$ is purely infinite, there exist $x,y\in R$ such that $b=xay$, and hence $b= (exe)a(eye)$ with
$exe,eye\in eRe$. This shows that $eRe$ is purely infinite.
\end{proof}

We next study the inheritance of pure infiniteness by matrix rings. In general, matrix rings over
purely infinite rings need not be purely infinite, since otherwise pure infiniteness and strong
pure infiniteness would be the same (recall Lemma \ref{EveryPropInfImpliesPi} and Remark
\ref{spiIMPLIESpi}). We shall prove that pure infiniteness passes to matrix rings in certain
circumstances, and strong pure infiniteness in much wider circumstances. First we prove the
following useful lemma.

\begin{lemma} \label{MatricesAndPropInf} Let $R$ be a purely infinite unital ring. If $1\in R$ is properly
infinite, then ${\MM}_n(R)$ is properly purely infinite for all $n\in \NN$.
\end{lemma}

\begin{proof}
As $1$ is properly infinite, $(R\oplus R)_R$ is isomorphic to a direct summand of $R_R$, and so
$R_R^n$ is isomorphic to a direct summand of $R_R$ for every $n\in \NN$. Hence, there are
idempotents $f_n\in R$ such that $f_nR \cong R_R^n$. Thus, ${\MM}_n(R)\cong f_nRf_n$, which is
purely infinite by Proposition \ref{PiCorners}. But $\MM_2(\MM_n(R)) \cong \MM_{2n}(R)$ is purely
infinite for all $n$, so we are done by Lemma \ref{EveryPropInfImpliesPi}.
\end{proof}

\begin{proposition} \label{pinftomatrices} Let $R$ be a ring with local units. If $R$ is properly
purely infinite, then so is every matrix ring $\MM_n(R)$.  \end{proposition}

\begin{proof} Any nonzero matrix $a\in \MM_n(R)$ lies in $\MM_n(eRe)$ for some nonzero idempotent $e\in
R$. Since $R$ is properly purely infinite, $e$ is properly infinite, and so Proposition
\ref{PiCorners} and Lemma \ref{MatricesAndPropInf} together imply that $\MM_n(eRe)$ is properly
purely infinite. Consequently, $a$ is properly infinite in $\MM_n(eRe)$, and hence also in
$\MM_n(R)$.
\end{proof}

Recall that a ring $R$ is \emph{irreducible} provided $R\ne 0$ and the intersection of any two
nonzero ideals of $R$ is nonzero.

\begin{lemma}\label{PrimePiIdempotentImplies1Infinite} Let $R$ be an irreducible, purely infinite, unital
 ring in which
every nonzero ideal contains nonzero idempotents. Then $1\in R$ is infinite.
\end{lemma}

\begin{proof}
If $R$ is simple, the result is clear from \cite[Theorem 1.6]{AGP}. Suppose then that there exists
a nontrivial ideal $I$ in $R$, and pick a nonzero idempotent $e\in I$. Then $0\ne ReR\subsetneq
R$. In particular, we have $e\neq 0,1$, and since $R$ is irreducible, $ReR\cap R(1-e)R \neq 0$.

By hypothesis, there exists a nonzero idempotent $f\in ReR\cap R(1-e)R$. Use now the pure
infiniteness of $R$ to obtain elements $x,y,z,w\in R$ with $f=xey=z(1-e)w$. This implies that
$f\lesssim e$ and $f\lesssim 1-e$, and so there exist nonzero orthogonal idempotents $f_1,f_2\in
R$ such that $f\sim f_1 \sim f_2$. Take $f_3= 1-(f_1+f_2)$, so that $1=f_1\oplus f_2 \oplus f_3$.

Since $f_1\sim f_2\le f_2+f_3$, we have $f_1\in R(f_2+f_3)R$, and hence $R(f_2+f_3)R =R$. Using
the pure infiniteness of $R$ once more, we obtain $u,v\in R$ with $1= u(f_2+f_3)v$, whence
$1\lesssim f_2+f_3<1$. Therefore $1$ is infinite.
\end{proof}

\begin{proposition}\label{ExchangePi1infinite} Let $R$ be a purely infinite exchange ring. Then
every nonzero idempotent in $R$ is properly infinite.
\end{proposition}

\begin{proof}
If $e$ is a nonzero idempotent in $R$, then $eRe$ is an exchange ring, and $eRe$ is purely
infinite by Proposition \ref{PiCorners}. Since it suffices to prove that $e$ is properly infinite
in $eRe$, we may replace $R$ by $eRe$. Thus, we may assume that $R$ is a nonzero unital ring and
$e=1$.

By Corollary \ref{everyquotientfinite}, it is enough to show that the set
 $$\C =\{I \mid I\text{ is a proper ideal of }R\text{ and } 1+I \in R/I\text{ is finite}\}$$
 is empty. Assume, to the contrary, that $\C$ is nonempty, and observe that $\C$ is an inductive
set. By Zorn's Lemma, there exists a maximal element $M\in \C$. Since $R/M$ is a purely infinite
exchange ring (Lemma \ref{pinftoidealquo}), we may replace $R$ by $R/M$. Thus, there is no loss of
generality in assuming that $1\in R$ is finite and that $1+J\in R/J$ is infinite for all proper
nonzero ideals $J$ of $R$.

We next claim that $J(R)=0$. Otherwise, $R/J(R)$ is a proper quotient of $R$ in which $\onebar$ is
infinite, so that $\onebar \sim \onebar\oplus e$ for some nonzero idempotent $e\in R/J(R)$. But
$e$ lifts to an idempotent $f\in R$ (because $R$ is an exchange ring), and $\onebar \sim
\onebar\oplus \fbar$ implies $1\sim 1\oplus f$, which contradicts the finiteness of $1\in R$.
Hence, $J(R)=0$ as claimed.

Since $R$ is now a semiprimitive exchange ring, we see, as in the proof of Proposition
\ref{SemiprimitiveExchangePi}, that every nonzero (right) ideal of $R$ contains a nonzero
idempotent.

Moreover, we will show that $R$ is irreducible. Suppose that $I$ and $J$ are nonzero ideals of $R$
with $I\cap J=0$. In particular, $(I+J)/I\cong J$. Then the image of $1$ is infinite in $R/I$ and
in all nonzero quotients of $R/I$, whence $1+I$ is properly infinite in $R/I$ by Corollary
\ref{everyquotientfinite}. Therefore, all nonzero idempotents in $R/I$ are properly infinite, by
Lemma \ref{MatricesAndPropInf}. In particular, $J$ contains properly infinite idempotents, which
contradicts the assumption that $1$ is finite in $R$. Hence, $R$ is irreducible.

Now we are in position to apply Lemma \ref{PrimePiIdempotentImplies1Infinite} to obtain that $1$
is infinite in $R$, a contradiction. Therefore the original collection $\C$ must be empty, as
desired.
\end{proof}

\begin{theorem}\label{ExchangePiMatrices} Let $R$ be an s-unital purely infinite exchange ring.
Then ${\MM}_n(R)$ is properly purely infinite for every $n\in \NN$.
\end{theorem}

\begin{proof}
By Lemma \ref{sunitalexchange}, $R$ has local units, and so $R$ is a directed union of corners
$eRe$. Hence, each matrix ring $M_n(R)$ is a directed union of subrings $M_n(eRe)$. The rings
$eRe$ are purely infinite exchange rings by Proposition \ref{PiCorners}, and it suffices to show
that the rings $M_n(eRe)$ are all properly purely infinite. Thus, there is no loss of generality
in assuming that $R$ is unital. By Proposition \ref{ExchangePi1infinite}, $1\in R$ is properly
infinite, and so we are done by Lemma \ref{MatricesAndPropInf}.
\end{proof}

\begin{corollary} \label{exchangepinfstrong} Let $R$ be an s-unital exchange ring. Then
$R$ is purely infinite if and only if it is properly purely infinite.
\end{corollary}

Recall that for unital rings, a property $\P$ is Morita invariant if and only if $\P$ passes to
corners by full idempotents and to matrices. Hence, we immediately obtain the following from
Proposition \ref{PiCorners} and Theorem \ref{ExchangePiMatrices}.

\begin{corollary}\label{moritainvariant} Pure infiniteness is a Morita invariant property for unital
exchange rings. \end{corollary}

We close this section by showing that pure infiniteness is also a
Morita invariant property for the class of exchange rings with
local units. We begin by showing that more general (type of)
corners than the ones considered previously also inherit the
property of being purely infinite.

\begin{definition}
{\rm Let $R$ be a ring and let $a\in R$. The \emph{local ring} of
$R$ \emph{at} $a$ is defined as $R_a=aRa$, with sum inherited from
$R$, and product given by $axa\cdot aya= axaya$.}
\end{definition}

The use of local rings at elements allows to overcome the lack of
a unit element in the original ring, and to translate problems
from a nonunital context to a unital one. We refer the reader to
\cite{GS} for a fuller account on transfer of various properties
between rings and their local rings at elements. Notice that if
$e$ is an idempotent in the ring $R$, then the local ring of $R$
at $e$ is just the corner $eRe$.

For any ring $R$, denote $R^1=R\oplus\mathbb{Z}$, which becomes a
unital ring under componentwise addition, product given by
$(x,m)(y,n)=(xy+my+nx,nm)$, and unit $(0,1)$. Clearly $R$ embeds
into $R^1$ as an ideal. Part of the proof below follows the lines of Proposition~\ref{PiCorners}.

\begin{theorem}\label{localpurinf}  Let $R$ be a ring.
\begin{enumerate}[{\rm (i)}]
\item If $R$ is purely infinite then, for every $a\in R$, the
local ring of $R$ at $a$ is  purely infinite.
\item Suppose that $R$ is s-unital. If for every $a\in R$ the ring $R_a$ is purely
infinite, then $R$ is purely infinite.
\end{enumerate}
\end{theorem}
\begin{proof}
(i). Suppose first that $I$ is an ideal of $R_a$ such that $R_a/I$
is a division ring.  In this case, $I$ is a maximal ideal of
$R_a$. Consider $R^1IR^1$, the ideal of $R$ generated by $I$, and
let $aua+I$ be the identity element in  $R_a/I$, so that
\[
auaxa\equiv axaua\equiv axa \pmod{I} \text{ for all }x\in R\,.
\]
In particular,
$auaua-aua=y\in I$. Note that $auau\notin R^1IR^1$ because
otherwise $aua+I = (aua+I)^4= auauauaua+I=0$, a contradiction.

Denote by $\pi\colon R\to R/R^1IR^1$ the natural quotient
map. Then the nonzero element $e=\pi (auau)$ is an idempotent in
$R/R^1IR^1$. Indeed,
\begin{align*}
 e^2=\pi(auau)\, \pi(auau) & =
  \pi(auauauau)= \\ \pi((aua+y)uau)
 & =  \pi(auauau)=\pi((aua+y)u)=
\pi(auau)=e\,.
\end{align*}
In particular, $\pi(auau)$ cannot be in the Jacobson radical
of $R/R^1IR^1$. Hence, there exists a (left) primitive ideal $P$
of $R$ such that $R^1IR^1\subseteq P$ and $auau\notin P$.
Notice that $auaua\notin P$. (Otherwise, taking into account that
$auaua-aua\in I\subseteq P$ we see that $aua\in P$, which is not possible.)
Maximality of $I$ in $R_a$ now entails $I=P\cap R_a$.

We use $\pi'$ to denote the quotient map $R\to R/P$ and $e'=\pi'(auau)$, which still is a nonzero idempotent. It is clear that $R_a/I\cong (R/P)_{\pi'(a)}$, via the isomorphism
\[
\varphi\colon R_a/(P\cap R_a) \to  (R/P)_{\pi'(a)}
\]
given by $\varphi(x+(P\cap R_a))= x+P$.

Further, there is also an isomorphism
\[
\psi\colon e'(R/P)e'\to (R/P)_{\pi'(a)}\,,
\]
given by $\psi(y) = y\pi'(a)$. This yields
\[
e'(R/P)e' \cong (R/P)_{\pi'(a)}\cong R_a/I\,.
\]
Hence, $R/P$ has a corner which is a division ring, contradicting Lemma~\ref{PrimePiNoDivCorners}.

Next, suppose $aba\in aRa$ and let $aca\in aRa \cdot aba\cdot aRa = aRabaRa$. (We may assume $c\in RabaR$.)
Since $R$ is purely infinite, there exist $x$, $y \in R$ such that
$c=xabay$, whence  $aca= axabaya=axa \cdot aba \cdot aya$.

(ii). We now assume $R$ to be s-unital and that all local rings of $R$ are purely infinite. If $R$ is unital, there is nothing to prove.

Suppose that $R/I$ is a division ring, for an ideal $I$ of
$R$. Let $\pi\colon R\to R/I$ denote the natural quotient map, and observe that $\pi$ restricts to a surjective ring homomorphism $R_u\to (R/I)_{\pi(u)}=R/I$. Thus $R_u$ has  a quotient which is a division ring, contradicting our hypothesis.

Next, take $b\in R$ and $a\in RbR$.
Let $x$ in $R$ be such that $a=ax=xa$ and $b=bx=xb$. In
particular,  $a$, $b \in R_x$. Apply that $R_x$ is a purely infinite
ring to find $xrx$, $xsx\in R_x$ with $a=xrx\cdot xbx \cdot
xsx= xrxbxsx$.
\end{proof}

\begin{corollary} \label{LocalUnitsCorners}
Let $R$ be a ring with local units. Then $R$ is purely infinite if
and only if every corner of $R$ is purely infinite.
\end{corollary}

Using Theorem~\ref{localpurinf}, the following becomes immediate (although it has been already proved in Lemma~\ref{pinftoidealquo}).
\begin{corollary}\label{IdealPurelyInfinite}
Let $I$ be an s-unital ideal of a purely infinite ring $R$. Then
$I$ is purely infinite.
\end{corollary}
\begin{proof}
For $y\in I$, choose $a\in I$ such that $y=ay=ya$. Then
$yRy=yaRay\subseteq yIy\subseteq yRy$, and thus $I_y=R_y$. Now apply
Theorem~\ref{localpurinf}.
\end{proof}

We next recall the notion of Morita equivalence for idempotent
rings (a ring $R$ is said to be \emph{idempotent} if $R^2=R$).

Let $R$ and $S$ be two rings, $_R N_S$ and $_S M_R$ two bimodules
and $(-, -)\colon N \times M \to R$, $[  -, -]\colon M \times N \to S$ two
maps. Then the following conditions are equivalent:
\begin{enumerate}
\item[(i)] $\left(\begin{matrix}
 R & N \cr M & S\end{matrix}\right)$
  is a ring with componentwise sum and product given by:
\[
 \left(\begin{matrix}r_1 & n_1 \cr m_1 & s_1\end{matrix}\right)
          \left(\begin{matrix} r_2 & n_2 \cr m_2 & s_2\end{matrix}\right) =
         \left(\begin{matrix}
         r_1r_2 +(n_1, m_2) & r_1n_2 +n_1s_2 \cr
          m_1r_2 +s_1m_2 & [  m_1, n_2] +s_1s_2
          \end{matrix}\right)
\]
\item[(ii)] $[ -, -]$ is $S$-bilinear and $R$-balanced,  $( -, -)$
is $R$-bilinear and $S$-balanced and the following associativity
conditions hold:
\[
(n, m)n^{\prime} = n [m, n^{\prime}] \quad \hbox{and} \quad
            [m, n] m^{\prime} = m (n, m^{\prime})\,,
\]
for all $m$, $m'\in M$ and $n$, $'\in N$.

That $[ -, -]$ is $S$-bilinear and $R$-balanced and that $( -, -)$ is $R$-bilinear and
$S$-balanced is equivalent to having bimodule maps $\varphi \colon N \otimes_S M \to R$ and  $\psi\colon M \otimes_R N \to S$, given by
\[
\varphi (n \otimes m) = (n, m) \quad \hbox{and} \quad \psi (m \otimes n) = [m, n]
\]
so that the associativity conditions above read
\[
\varphi (n \otimes m) n^\prime= n \psi (m \otimes n^\prime) \quad \hbox{and} \quad
       \psi (m \otimes n) m^\prime = m \varphi (n\otimes m^\prime)\,.
\]
\end{enumerate}
A \emph{Morita context} is a sextuple $(R, S, N, M, \varphi,
\psi)$ satisfying one of the (equivalent) conditions given above. The associated ring (in condition (i))
is called the \emph{Morita ring of the context}. By abuse  of
notation we will write $(R, S, N, M)$ instead  of $(R, S, N, M,
\varphi, \psi)$ and  will identify $R$, $S$, $N$ and $M$ with their natural images in
the Morita ring associated to the context. The Morita context is said to be \emph{surjective} if the maps $\varphi$ and $\psi$ are
both surjective.

In classical Morita theory, it is shown that two rings with
identity $R$ and $S$ are Morita equivalent (i.e., $R$-Mod and
$S$-Mod are equivalent categories) if and only if there exists a
surjective Morita context $(R, S, N, M, \varphi, \psi)$. The
approach to  Morita theory for rings without identity by means of
Morita contexts appears in a number of papers (see~\cite{GarciSimon} and the references therein) in which many
consequences are obtained from the existence of a surjective Morita context
for two rings $R$ and $S$.

For an idempotent ring $R$ we denote by $R$-Mod the full
subcategory of the category of all left $R$-modules whose objects
are the ``unital" nondegenerate modules. Here, a left $R$-module
$M$ is said to be \emph{unital} if $M=RM$, and $M$ is said to be
\emph{nondegenerate} if, for $m\in M$, $Rm=0$ implies $m=0$. Note
that, if $R$ has an identity, then $R$-Mod is the usual category
of left $R$-modules.

It is shown in~\cite[Theorem]{Ky} that, if $R$ and $S$ are
arbitrary rings having a surjective Morita context, then the
categories $R$-Mod and $S$-Mod are equivalent. The converse direction is proved in~\cite[Proposition 2.3]{GarciSimon} for idempotent rings, yielding the following theorem.

\begin{theorem}Let $R$ and $S$ be two idempotent rings. Then the categories $R$-Mod and
$S$-Mod are equivalent if and only if there exists a surjective
Morita context $(R, S, N, M)$.
\end{theorem}

Given two idempotent rings $R$ and $S$, we will say that they are
\emph{Morita equivalent} if the categories $R$-Mod and $S$-Mod are equivalent.

The following result states that purely infiniteness is a Morita
invariant property for exchange rings with local units.

\begin{theorem} Let $R$ and $S$ be rings with local units that are Morita equivalent. Then $R$ is purely infinite and exchange if and only if
$S$ is purely infinite and exchange.
\end{theorem}
\begin{proof}
Rings with local units are clearly idempotent, and in this case the exchange property
is Morita invariant, as shown in~\cite[Theorem 2.3]{AGS}. Let $(R, S, N, M)$ be a surjective Morita context and assume that $R$ is a purely infinite exchange ring. We will show that
$S_e$ is a purely infinite ring for every idempotent $e$ in $S$. The result then follows by applying Corollary~\ref{LocalUnitsCorners}.

Since $e\in S=MN$, we can find $x_1, \dotsc, x_n\in M$, $y_1,
\dotsc, y_n \in N$ satisfying $e=\sum_{i=1}^n x_iy_i$.
Put $\mathbf{x}=(x_1, \dots , x_n)$, $\mathbf{y}= (y_1, \dots ,
y_m)$. Then $e=\mathbf{x}\cdot\mathbf{y}^t$, and
we may assume $x_i=ex_i$ and $y_i=y_ie$, for every $i\in \{1,
\dots, n\}$.

Consider the map
\[
\begin{matrix}
\varphi\colon & M_n (R)_{\mathbf{y}^t\cdot\mathbf{x}} & \to & S_e \cr & (\mathbf{y}^t\cdot\mathbf{x})a (\mathbf{y}^t\cdot\mathbf{x}) &\mapsto &
\mathbf{x} a \mathbf{y}^t\,,
\end{matrix}
\]
which is easily seen to be a ring isomorphism. Since matrix rings over purely infinite exchange rings with local units are again purely infinite (Theorem~\ref{ExchangePiMatrices}) we obtain (via Theorem~\ref{localpurinf} or Proposition~\ref{PiCorners}) that $S_e$ is a purely infinite ring.
\end{proof}

\section{Purely infinite rings with refinement for idempotents}

There exist rings, such as the Leavitt path algebras we discuss in the following section, which
are not exchange rings but have some similar properties, such as refinement for orthogonal sums of
projections. Our particular goal in this section is to extend Theorem \ref{ExchangePiMatrices} to
a class of such rings. For efficient use of refinement arguments, we work with the monoids $V(R)$
(recall Definitions \ref{VRdef}).

For the reader's convenience, we collect here some standard concepts concerning abelian monoids
that will be needed below.

\begin{definitions} {\rm Let $V$ be an abelian monoid. We say that $V$ is a \emph{refinement monoid}
provided that for any $x_1,x_2,y_1,y_2 \in V$ with
$x_1+x_2= y_1+y_2$, there exist $z_{ij}\in V$ such that $z_{i1}+z_{i2} =x_i$ for $i=1,2$ and
$z_{1j}+z_{2j}= y_j$ for $j=1,2$. These equations can be conveniently displayed in the form of a
\emph{refinement matrix}:
 $$\bordermatrix{ &y_1 &y_2 \cr  x_1 &z_{11} &z_{12} \cr  x_2 &z_{21} &z_{22} \cr}$$
 Refinements of equations with more terms follow by induction. Moreover, refinement implies the
 \emph{Riesz decomposition property}: whenever $x,y_1,y_2\in V$ with $x\le y_1+y_2$, there exist
 $x_i\in V$ such that $x=x_1+x_2$ and $x_i\le y_i$ for $i=1,2$.
 When $R$ is an exchange ring, $V(R)$ has refinement (e.g., \cite[Corollary 1.3]{AGOP}). It is easily checked that if
$V$ is a refinement monoid and $I$ is an ideal, then both $I$ and $V/I$ are also refinement monoids (e.g.,~\cite[Proposition 7.8]{Gary}).

An element $u\in V$ is called \emph{irreducible} provided
 \begin{enumerate}
  \item $u$ is not a unit;
   \item whenever
$a,b\in V$ and $u=a+b$, either $a$ or $b$ is a unit.
 \end{enumerate}
In case $V$ is conical, the definition simplifies because $0$ is the only unit. In this case, $u$
is irreducible if and only if
 \begin{enumerate}
 \item[(1$'$)] $u\ne 0$;
\item[(2$'$)] whenever $a,b\in V$ and $u=a+b$, either $a=0$ or $b=0$.
 \end{enumerate}
 Condition (2$'$)
extends by induction to sums of more than two terms: if $u= a_1+\cdots +a_n$ for some $a_i\in V$,
then there is an index $j$ such that $a_j=u$ and $a_i=0$ for all $i\ne j$.
 }\end{definitions}

\begin{lemma} \label{Fred} {\rm \cite[1.9]{Weh}} Let $V$ be a refinement monoid, $x,y,z\in V$, and $n\in\NN$.
If $nx=y+z$, then $x= x_0+ \cdots+ x_n$ for some $x_i\in V$ such that $x_1+ 2x_2+ \cdots+ nx_n =y$
and $nx_0+ (n-1)x_1+ \cdots+ x_{n-1} =z$. \end{lemma}

\begin{corollary} \label{nalessblift} Let $V$ be a refinement monoid, $x,u\in V$, and $I$ an ideal of
$V$. If $n\xbar \le \ubar$ in $V/I$ for some $n\in \NN$, then $x=x'+c$ for some $x'\in V$ and
$c\in I$ such that $nx'\le u$.
\end{corollary}

\begin{proof} By hypothesis, $nx\le u+a$ for some $a\in I$, whence $nx= y+z$ for some $y\le u$ and
$z\le a$. By Lemma \ref{Fred}, $x= x_0+ \cdots+ x_n$ for some $x_i\in V$ such that $x_1+ 2x_2+
\cdots+ nx_n =y$ and $nx_0+ (n-1)x_1+ \cdots+ x_{n-1} =z$. For $i<n$, we have $x_i\le z\le a$, and
so $x_i\in I$. Thus, the lemma is satisfied with $x'=x_n$ and $c= x_0+ \cdots+ x_{n-1}$.
\end{proof}

We next establish a lifting property which is the analog for refinement monoids of Effros's
lifting property for decompositions of projections modulo ideals in AF C$^*$-algebras \cite[Lemma
9.8]{Eff}.

\begin{lemma} \label{decomplift} Let $V$ be a refinement monoid,
$u\in V$, and $I$ an ideal of $V$. Suppose that $\ubar= n_1\xbar_1 +\cdots+ n_r\xbar_r$ in $V/I$
for some $n_i\in\NN$ and $x_i\in V$. Then there exist $y_i\in V$ such that $n_1y_1 +\cdots+ n_ry_r
\le u$ and $\ybar_i= \xbar_i$ in $V/I$ for all $i$.  \end{lemma}

\begin{proof}  By hypothesis, $u+a= n_1x_1+ \cdots+ n_rx_r +b$ for some $a,b\in I$. Refine this
equation, obtaining a refinement matrix
$$\bordermatrix{ &n_1x_1 &\cdots &n_rx_r &b \cr u &z_{11} &\cdots &z_{1r} &z_{1,r+1} \cr  a
&z_{21} &\cdots &z_{2r} &z_{2,r+1}}$$
 for some $z_{ij} \in V$.
 Each $z_{2i} \le a$, so $z_{2i} \in I$ and $n_i\xbar_i=
\zbar_{1i}$ in $V/I$ for $i\le r$. By Corollary \ref{nalessblift}, each $x_i= y_i+c_i$ for some
$y_i\in V$ and $c_i\in I$ such that $ny_i\le z_{1i}$. Thus,
 $$n_1y_1 +\cdots+ n_ry_r \le z_{11} +\cdots+
z_{1r} \le u,$$
 and $\ybar_i= \xbar_i$ in $V/I$ for all $i$. \end{proof}

\begin{definition} {\rm Let $V$ be an abelian monoid. An element $u\in V$ is \emph{abelian} (or: \emph{$u$
has index $1$}) provided the only elements $a\in V$ for which $2a\le u$ are the units in $V$.
 }\end{definition}

\begin{lemma} \label{abelianorderunit} Let $V$ be a refinement monoid with an abelian order-unit $u$.
 \begin{enumerate}[{\rm (i)}]
\item If $I$ is an ideal of $V$, then the order-unit $\ubar\in V/I$ is abelian.
\item If $M$ is a maximal ideal of $V$, then the order-unit $\ubar\in V/I$ is irreducible.
 \end{enumerate}
\end{lemma}

\begin{proof} (i) Suppose that $2\abar\le \ubar$ in $V/I$, for some $a\in V$. By Corollary
\ref{nalessblift}, $a= a'+c$
for some $a'\in V$ and $c\in I$ such that $2a'\le u$. Since $u$ is abelian, $a'$ is a unit, and so
$a'\in I$. Consequently, $a\in I$ and $\abar= 0$ in $V/I$. Therefore $\ubar$ is abelian in $V/I$.

(b) The refinement monoid $V/M$ is conical and simple, and its order-unit $\ubar$ is abelian by
part (a). Thus, after passing to $V/M$, there is no loss of generality in assuming that $V$ is
conical and simple, and that $M=\{0\}$.

Suppose that $u=a+b$ for some $a,b\in V$ with $b\ne 0$. By simplicity, $a\le nb$ for some
$n\in\NN$, and so $a= a_1+ \cdots+ a_n$ for some $a_i\le b$. Since $2a_i\le a+b= u$, we have
$a_i=0$ for all $i$, and thus $a=0$. Therefore $u$ is irreducible. \end{proof}

We can now prove a monoid version of \cite[Proposition 5.7]{PeRo}. Portions of our proof are
adapted from \cite[Lemma 3.4]{Bro}.

\begin{theorem} \label{23div} Let $V$ be a refinement monoid and $u\in V$ such that $\ubar$ is
not irreducible in $V/I$
for any ideal $I$ of $V$. Then there exist $x,y\in V$ such that $u=2x+3y$.  \end{theorem}

\begin{proof} Since we may work in the ideal generated by $u$, we may assume
that $u$ is an order-unit in $V$.

Let $J$ be the ideal of $V$ generated by the set $X= \{x\in V\mid 2x\le u\}$. We claim that
$\ubar$ is abelian in $V/J$. If $2\abar\le \ubar$ in $V/J$ for some $a\in V$, then Corollary
\ref{nalessblift} shows that, after possibly replacing $a$ by some element congruent to it modulo
$J$, we may assume that $2a\le u$. Then $a\in X$ and $\abar=0$ in $V/J$, verifying that $\ubar$ is
indeed abelian in $V/J$. Thus, we must have $J=V$.

Now $u\in J$, so $u\le x_1+ \cdots+ x_n$ for some $x_i\in X$. Hence, $u= u_1+ \cdots+ u_n$ for
some $u_i \le x_i$, and each $u_i\in X$. For $k=1,\dots,n$, let $J_k$ denote the ideal of $V$
generated by $\{u_1,\dots,u_k\}$. We claim that each $J_k$ can be generated by an element of $X$.
This is clear for $J_1$, which is generated by $u_1$.

Suppose that we have an element $v_k\in X$ which generates $J_k$, for some $k<n$. Write $u=
2v_k+w$ for some $w\in V$, and note that $2\ubar_{k+1} \le \ubar= \wbar$ in $V/J_k$. By Corollary
\ref{nalessblift}, $u_{k+1} =u'_{k+1} +c$ for some $u'_{k+1} \in V$ and $c\in J_k$ such that
$2u'_{k+1} \le w$. Set $v_{k+1}= v_k+ u'_{k+1}$. Since $v_{k+1} \le v_k+ u_{k+1}$ and $2v_{k+1}
\le 2v_k+w= u$, we see that $v_{k+1} \in J_{k+1}$ and $v_{k+1} \in X$. On the other hand, $v_k\le
v_{k+1}$ and $u_{k+1} \le v_{k+1}+c\le v_{k+1}+mv_k\le (m+1)v_{k+1}$ for some $m\in\NN$. It
follows that $J_{k+1}$ is generated by $v_{k+1}$, verifying the induction step of our claim.

The case $k=n$ of the claim provides an element $x=v_n\in X$ which generates the ideal $J_n$. By
construction, $u\in J_n$, so $J_n=V$, and thus $x$ is an order-unit in $V$. Since $x\in X$, we
also have $u=2x+y$ for some $y\in V$.

Now $y\le mx$ for some $m\in\NN$, and so $mx= y+z$ for some $z\in V$. By Lemma \ref{Fred}, $x=
x_0+ \cdots+ x_m$ for some $x_i\in V$ such that $x_1+ 2x_2+ \cdots+ mx_m =y$. Set
 $$r = \sum_{i=1}^m \lfloor i/2\rfloor x_i \qquad\qquad\qquad s = \sum^m_{\substack{i=1\\
 i\text{\ odd}}} x_i \qquad\qquad\qquad  t =
 \sum^m_{\substack{i=0\\ i\text{\ even}}} x_i\,,$$
 so that $y=2r+s$ and $x=s+t$. Therefore $u=2x+y=
2(r+t)+3s$.  \end{proof}

Theorem \ref{23div} immediately yields a generalization of \cite[Proposition 5.7]{PeRo} to the
nonseparable case, and a corresponding result for exchange rings, as follows.

\begin{corollary} Let $A$ be a C$^*$-algebra with real rank zero, and $p\in A$ a projection such that
the corner $pAp$ has no $1$-dimensional representations. Then $\MM_2(\CC)\oplus \MM_3(\CC)$ is
isomorphic to a unital sub-C$^*$-algebra of $pAp$. \end{corollary}

\begin{proof} If $I$ is any (closed) ideal of $A$ not containing $p$, then $pAp/pIp$ is not
1-dimensional by hypothesis. Since $A$ has real rank zero, it follows that $pAp/pIp$ has
projections different from $0$ and $\pbar$, and so $\pbar$ is a sum of two nonzero orthogonal
projections. This shows that $\overline{[p]}$ is not irreducible in $V(A)/V(I)$. Since all ideals
of $V(A)$ have the form $V(I)$ for closed ideals $I$ of $A$, we conclude that $\overline{[p]}$ is
not irreducible in any quotient of $V(A)$.

Theorem \ref{23div} now implies that $[p]= 2x+3y$ for some $x,y\in V(A)$. Hence,
 $$p=r_1+r_2+s_1+s_2+s_3$$
  for some pairwise orthogonal projections $r_i$ and $s_j$ with $r_1\sim r_2$ and $s_1\sim s_2\sim
s_3$. The corner $(r_1+r_2)A(r_1+r_2)$ then contains a complete set of $2\times 2$ matrix units,
and so has a unital subalgebra isomorphic to $\MM_2(\CC)$. Similarly,
$(s_1+s_2+s_3)A(s_1+s_2+s_3)$ has a unital subalgebra isomorphic to $\MM_3(\CC)$. Therefore
$\MM_2(\CC)\oplus \MM_3(\CC)$ embeds unitally in
 $$(r_1+r_2)A(r_1+r_2)\oplus
(s_1+s_2+s_3)A(s_1+s_2+s_3),$$
 which is a unital subalgebra of $pAp$.  \end{proof}

\begin{corollary} Let $R$ be an exchange ring, and $e\in R$ an idempotent such that no quotient of
$eRe$ is a division ring. Then there exist unital rings $R_2$ and $R_3$ such that
$\MM_2(R_2)\oplus \MM_3(R_3)$ is isomorphic to a unital subring of $eRe$.  \end{corollary}

\begin{proof} This is analogous to the previous proof. If $I$ is any ideal of $R$ not containing
$e$, then no quotient of $eRe/eIe$ is a division ring, and so $eRe/eIe$ cannot be a local ring.
Since $eRe/eIe$ is an exchange ring, it thus must contain an idempotent different from $0$ and
$1$, from which it follows that $\overline{[e]}$ is not irreducible in $V(R)/V(I)$. Applying
Theorem \ref{23div}, we get $e= f_1+f_2+g_1+g_2+g_3$ for some pairwise orthogonal idempotents
$f_i$ and $g_j$ with $f_1\sim f_2$ and $g_1\sim g_2\sim g_3$. Consequently, there are matrix units
in appropriate corners yielding a unital subring of $eRe$ of the desired form.  \end{proof}

Our main use of Theorem \ref{23div} is to extend Theorem \ref{ExchangePiMatrices} to purely
infinite rings with refinement for idempotents, as follows.

\begin{theorem} \label{ideminpinfwithref} Let $R$ be a purely infinite ring, and assume that
$V(R)$ is a refinement
monoid. If $e\in R$ is an idempotent, and $\overline{[e]}$ is not irreducible in any quotient of
$V(R)$, then $\MM_n(eRe)$ is properly purely infinite for every $n\in\NN$. In particular, $e$ is
properly infinite.
 \end{theorem}

\begin{proof} Applying Theorem \ref{23div} to
$V(R)$, we obtain that $e= f_1+f_2+g_1+g_2+g_3$ for some pairwise orthogonal idempotents
$f_i,g_j\in R$ such that $f_1\sim f_2$ and $g_1\sim g_2\sim g_3$. Consequently, $p = f_1+g_1$ and
$q = f_1+f_2+g_1+g_2$ are idempotents in $R$ such that $e\in RpR$ and $qRq \cong \MM_2(pRp)$.
Since $R$ is purely infinite, there exist $x,y\in R$ such that $xpy=e$, whence $e\lesssim p$. This
means that $eRe$ is isomorphic to a corner of $pRp$, and so $\MM_2(eRe)$ is isomorphic to a corner
of $qRq$. In view of Proposition \ref{PiCorners}, $\MM_2(eRe)$ is purely infinite, whence Lemma
\ref{EveryPropInfImpliesPi}(ii) implies that $eRe$ is properly purely infinite. Therefore $e$ is
properly infinite, and we are done by Lemma \ref{MatricesAndPropInf}. \end{proof}

\begin{corollary} Let $R$ be a purely infinite ring with local units. Assume that $V(R)$ is a refinement
monoid, and that idempotents lift modulo all ideals of $R$. Then $\MM_n(R)$ is properly purely
infinite for every $n\in\NN$.
\end{corollary}

\begin{proof} As in the proof of Theorem \ref{ExchangePiMatrices}, there is no loss of generality in assuming
that $R$ is unital. Set $u= [1_R] \in V(R)$. In view of Theorem \ref{ideminpinfwithref}, we need
only show that $\ubar$ is not irreducible in any quotient of $V(R)$.

Suppose, to the contrary, that $V(R)$ has an ideal $I$ such that $\ubar$ is irreducible in
$V(R)/I$. Let $E$ be the set of those idempotents $e\in R$ for which $[e]\in I$, and let $J$ be
the ideal of $R$ generated by $E$. If $J=R$, then $1= a_1e_1b_1+ \cdots+ a_ne_nb_n$ for some
$a_i,b_i\in R$ and $e_i\in E$. But then $1\lesssim e_1\oplus \cdots\oplus e_n$, implying $u\le
[e_1] +\cdots+ [e_n]$ in $V(R)$ and so $u\in I$, contradicting the assumption that $\ubar\in
V(R)/I$ is nonzero. Thus, $J\ne R$.

Now choose a maximal ideal $M$ of $R$ containing $J$. Then $R/M$ is a purely infinite simple ring
(Lemma \ref{pinftoidealquo}), and so $R/M$ contains idempotents different from $0$ and $1$
\cite[Theorem 1.6]{AGP}. Pick such an idempotent, say $p$, and lift it to an idempotent $q\in R$.
Then $\qbar$ and $\onebar-\qbar$ are both nonzero in $R/M$, and so $q,1-q\notin J$. Consequently,
$q,1-q\notin E$, whence $[q],[1-q] \notin I$. Since $[q]+[1-q]= u$, this contradicts the
assumption that $\ubar$ is irreducible in $V(R)/I$.

Therefore $\ubar$ is not irreducible in any quotient of $V(R)$, as desired.  \end{proof}

\section{Non-simple purely infinite Leavitt path algebras} \label{leavittsec}

Leavitt path algebras $L_K(E)$ of row-finite graphs have been recently introduced in \cite{AA1}
and \cite{AMP}. They have become a subject of significant interest, both for algebraists and for
analysts working in C$^*$-algebras. The Cuntz-Krieger algebras $C^*(E)$ (the C$^*$-algebra counterpart
of these Leavitt path algebras) are described in \cite{R}. The algebraic and analytic theories
share some striking similarities, as well as some distinct differences (see, e.g.,~\cite{Work} and~\cite{T}).

In the analytic context of graph C$^*$-algebras, the (not necessarily simple) purely infinite ones
were studied in \cite{HS}. In this section we will give the algebraic version of these results. In
fact, this can also be regarded as a natural follow up of the characterization of purely
infinite simple Leavitt path algebras that was carried out in \cite{AA2}.

We have chosen to restrict attention to row-finite graphs with (at most) countably many vertices, mainly to keep the paper down to a reasonable length. The more general setting of arbitrary uncountable row-finite graphs (using, e.g., the techniques from~\cite{leavittgoodearl}) will be pursued elsewhere.

First, we collect various notions concerning graphs, after which we define Leavitt path algebras.

\begin{definitions} {\rm
A (\emph{directed}) \emph{graph} $E=(E^0,E^1,r,s)$ consists of two countable sets $E^0$ and $E^1$
together with maps $r,s:E^1 \to E^0$. The elements of $E^0$ are called \emph{vertices} and the
elements of $E^1$ \emph{edges}. For $e\in E^1$, the vertices $s(e)$ and $r(e)$ are called the
\emph{source} and \emph{range} of $e$, respectively, and $e$ is said to be an \emph{edge from
$s(e)$ to $r(e)$}, represented by an arrow $s(e) \rightarrow r(e)$ when $E$ is drawn. If
$s^{-1}(v)$ is a finite set for every $v\in E^0$, then the graph is called \emph{row-finite}. If
$E^0$ is finite and $E$ is row-finite, $E^1$ must necessarily be finite as well; in this case we
say simply that $E$ is \emph{finite}. Here we will be concerned only with finite and row-finite
graphs.

A vertex which emits no edges is called a \emph{sink}. A \emph{path} $\mu$ in a graph $E$ is a
sequence of edges $\mu=e_1\dots e_n$ such that $r(e_i)=s(e_{i+1})$ for $i=1,\dots,n-1$. In this
case, $s(\mu)=s(e_1)$ and $r(\mu)=r(e_n)$ are the \emph{source} and \emph{range} of $\mu$,
respectively, and $n$ is the \emph{length} of $\mu$. We also say that $\mu$ is \emph{a path from
$s(e_1)$ to $r(e_n)$}, and we denote by $\mu^0$ the set of its vertices, i.e.,
$\{s(e_1),r(e_1),\dots,r(e_n)\}$.

\smallskip

If $\mu$ is a path in $E$, and if $v=s(\mu)=r(\mu)$, then $\mu$ is called a \emph{closed path based
at $v$}. If $s(\mu)=r(\mu)$
and $s(e_i)\neq s(e_j)$ for every $i\neq j$, then $\mu$ is called a \emph{cycle}. A graph which
contains no cycles is called
\emph{acyclic}.

An edge $e$ is an {\it exit} for a path $\mu = e_1 \dots e_n$ if there exists $i$ such that
$s(e)=s(e_i)$ and $e\neq e_i$. We say that $E$ satisfies \emph{Condition} (L) if every cycle in
$E$ has an exit. Let $M$ be a subset of $E^0$. A \emph{path in $M$} is a path $\alpha$ in $E$ with
$\alpha^0\subseteq M$. We say that a path $\alpha$ in $M$ has \emph{an exit in $M$} if there
exists $e\in E^1$ an exit for $\alpha$ such that $r(e)\in M$.

Recall that a \emph{closed simple path based at a vertex $v$} is a
path $\mu = e_1\cdots e_t$ such that $s(\mu)=r(\mu)=v$ and
$s(e_i)\neq v$ for all $2\leq i \leq t$. We denote the set of
closed simple paths based at $v$ by $CSP(v)$. Further, $E$ is said
to satisfy \emph{Condition} (K) if for each vertex $v$ on a closed
simple path there exist at least two distinct closed simple paths based at $v$.

We define a relation $\ge$ on $E^0$ by setting $v\ge w$ if there exists a path in $E$ from $v$ to
$w$. A subset $H$ of $E^0$ is called \emph{hereditary} if $v\ge w$ and $v\in H$ imply $w\in H$. A
hereditary set is \emph{saturated} if every vertex which feeds into $H$ and only into $H$ is again
in $H$, that is, if $s^{-1}(v)\neq \emptyset$ and $r(s^{-1}(v))\subseteq H$ imply $v\in H$. Denote
by $\mathcal{H}_E$ the set of hereditary saturated subsets of $E^0$.

We recall here some graph-theoretic constructions which will be of use. For a hereditary subset $H$ of $E^0$,
the \emph{quotient graph} $E/H$ is defined
as
$$(E^0\setminus H, \{e\in E^1|\ r(e)\not\in H\}, r|_{(E/H)^1}, s|_{(E/H)^1}),$$ and the \emph{restriction
graph} is
$$E_H=(H, \{e\in E^1|\ s(e)\in H\}, r|_{(E_H)^1}, s|_{(E_H)^1}).$$

\smallskip

The following definition (which is a particular case of that of \cite{BHRS}) will be used in our
main result: A nonempty subset $M\subseteq E^0$ is a \emph{maximal tail} if it satisfies the
following properties:
\begin{enumerate}
\item $E^0\setminus M$ is hereditary and saturated.
\item For every $v,w\in M$ there exists $y\in M$ such that $v\geq y$ and $w\geq y$.
\end{enumerate}
 }\end{definitions}

 Throughout this section, $K$ will denote an arbitrary base field.

 \begin{definitions} {\rm
The {\em Leavitt path $K$-algebra} $L_K(E)$, or simply $L(E)$ if the base field is understood, is
defined to be the $K$-algebra generated by the set $E^0\cup E^1\cup \{e^*\mid e\in E^1\}$ with the
following relations:
 \begin{enumerate}
\item $vw= \delta_{v,w}v$ for all $v,w\in E^0$.
\item $s(e)e=er(e)=e$ for all $e\in E^1$.
\item $r(e)e^*=e^*s(e)=e^*$ for all $e\in E^1$.
\item $e^*f=\delta _{e,f}r(e)$ for all $e,f\in E^1$.
\item $v=\sum _{e\in s^{-1}(v)}ee^*$ for every $v\in E^0$ that is not a sink.
 \end{enumerate}

The elements of $E^1$ are called \emph{real edges}, while for $e\in E^1$ we call $e^\ast$ a
\emph{ghost edge}.   The set $\{e^*\mid e\in E^1\}$ will be denoted by $(E^1)^*$.  We let $r(e^*)$
denote $s(e)$, and we let $s(e^*)$ denote $r(e)$.  If $\mu = e_1 \dots e_n$ is a path in $E$, we
denote by $\mu^*$ the element $e_n^* \dots e_1^*$ of $L(E)$. For any subset $H$ of $E^0$, we will
denote by $I(H)$ the ideal of $L(E)$ generated by $H$. Note that if $E$ is a finite graph, then
$L(E)$ is unital with $\sum _{v\in E^0} v=1_{L(E)}$.
 }\end{definitions}

\begin{lemma}\label{vInfiniteIdempotent} Let $E$ be a row-finite
graph. If $v\in E^0$ and $|CSP(v)|\geq 2$, then $v$ is a properly infinite idempotent in $L(E)$.
\end{lemma}

\begin{proof} Note that the relations (4) and (5) in the definition of $L(E)$ imply that for any
vertex $v\in E^0$, the elements $ee^*$ for $e\in s^{-1}(v)$ are pairwise orthogonal idempotents
with $r(e)= e^*e\sim ee^* \le v$. Hence, $v\sim \bigoplus_{e\in s^{-1}(v)} r(e)$ when $v$ is not a
sink. In particular, if $v,w\in E^0$ and $v\ge w$, then $v\gtrsim w$.

Let $e_1\dots e_m$ and $f_1\dots f_n$ be two different closed simple paths in $E$ based at $v$.
Then there is some positive integer $t$ such that $e_i=f_i$ for $i=1,\dots,t-1$ while $e_t\ne
f_t$. Thus, we have at least two different edges leaving the vertex $r(e_{t-1})= r(f_{t-1})$. We
compute that
 \begin{align*}
 v &= s(e_1) \gtrsim r(e_1) \gtrsim \cdots\gtrsim r(e_{t-1}) \gtrsim r(e_t)\oplus r(f_t)  \\
 &\gtrsim r(e_{t+1})\oplus r(f_{t+1}) \gtrsim \cdots\gtrsim r(e_m)\oplus r(f_n) \sim v\oplus v.
 \end{align*}
 Therefore $v$ is properly infinite. \end{proof}

The following result is the algebraic counterpart of \cite[Theorem 2.3]{HS}.

\begin{theorem}\label{LeavittPi} Let $L(E)$ denote the Leavitt path algebra of a row-finite graph
$E$. Then, the following conditions are equivalent:
\begin{enumerate}[{\rm (i)}]
\item Every nonzero right ideal of every quotient of $L(E)$ contains an infinite idempotent.
\item Every nonzero left ideal of every quotient of $L(E)$ contains an infinite idempotent.
\item $L(E)$ is properly purely infinite.
\item $L(E)$ is purely infinite.
\item Every vertex $v\in E^0$ is properly infinite as an idempotent in $L(E)$.
\item Every cycle in every maximal tail $M$ in $E$ has exits in $M$, and every vertex in $M$
connects to a cycle in $M$.
\item $E$ satisfies Condition {\rm(K)}, and every vertex in each maximal tail $M$ in $E$ connects
to a cycle in $M$.
\end{enumerate}
\end{theorem}

\begin{proof} Set $R=L(E)$, and observe that $R$ has local units.

(i) or (ii) $\Rightarrow$ (iii) $\Rightarrow$ (iv). These are Proposition
\ref{IdealInfIdemImpliesPi} and Lemma \ref{EveryPropInfImpliesPi}(i).

\smallskip

(iv) $\Rightarrow$ (v). By \cite[Proposition 4.4]{AMP}, $V(R)$ is a refinement monoid. Hence, by
Theorem \ref{ideminpinfwithref}, it suffices to show that $\overline{[v]}$ is not irreducible in
any quotient of $V(R)$. Now any ideal $I$ of $V(R)$ is of the form $V(I(H))$, where $I(H)$ is the
ideal of $R$ corresponding to some saturated hereditary subset $H\subseteq E^0$ \cite[Theorem
5.3]{AMP}. Moreover, we know that in this situation, $V(R)/I\cong V(R/I(H))\cong V(L(E/H))$. Since
there is nothing to do if $[v]\in I$, we may assume that $v\notin H$. By Lemma
\ref{pinftoidealquo}(i), $L(E/H)\cong R/I(H)$ is purely infinite, and so for this part of the
proof we may replace $R$ by $L(E/H)$. Thus, we need only show that $[v]$ is not irreducible in
$V(R)$, or equivalently, that $v$ is not a primitive idempotent.

Since $vRv$ is purely infinite (Proposition \ref{PiCorners}), it cannot be isomorphic to $K$ or to
a Laurent polynomial ring $K[x,x^{-1}]$. Hence, $v$ lies on at least one closed simple path, and
$CSP(v)$ cannot consist only of a single loop based at $v$. If $|CSP(v)| \ge 2$, then $v$ is
properly infinite by Lemma \ref{vInfiniteIdempotent}. In this case, $v$ is obviously not
primitive. If $|CSP(v)| =1$, then the unique closed simple path based at $v$ must pass through a
vertex $w\ne v$. Now $v\gtrsim w\gtrsim v$, whence $v\oplus v \lesssim v+w$ and so $\MM_2(vRv)$ is
isomorphic to a corner of $(v+w)R(v+w)$. In this case, Proposition \ref{PiCorners} and Lemma
\ref{EveryPropInfImpliesPi}(ii) imply that $vRv$ is properly purely infinite. Again, $v$ is
properly infinite and thus not primitive.

\smallskip

(v) $\Rightarrow$ (vi). Suppose that $M$ is a maximal tail in $E$, and that $\alpha$ is a cycle in
$M$ without exits in $M$. Pick $v\in \alpha^0$. The subset $H=E^0\setminus M$ is hereditary and
saturated, and $L(E)/I(H)\cong L(E/H)$ where $(E/H)^0=M$. Since being properly infinite is
preserved in quotients, $v$ is a properly infinite idempotent of $L(E/H)$.

On the other hand, because $M$ is a maximal tail and $\alpha$ does not have exits in $M$, the only
paths from $v$ to $v$ in $M$ are the powers of $\alpha$. It follows that
 $$vL(E/H)v \cong L(\alpha) \cong \MM_n(K[x,x^{-1}]),$$
 where $n= |\alpha^0|$. However, this ring does not contain properly infinite idempotents,
 contradicting the choice of $v$. Therefore every cycle in $M$ has exits in $M$.

Suppose now that there exists a vertex $v\in M$ not connecting to any cycle in $M$. The set
$H=\{w\in M \mid v\geq w\}$ is clearly hereditary and acyclic. In particular, $H$ contains no
paths from $v$ to $v$, from which we see that $vL(E/H)v \cong K$. This gives a contradiction as
before, and therefore every vertex in $M$ connects to a cycle in $M$.

\smallskip

(vi) $\Rightarrow$ (vii) is proved in \cite[Lemma 2.2]{HS}.

\smallskip

(vii) $\Rightarrow$ (i) Suppose that $J$ is a proper ideal of $R$. Because we have Condition (K),
$J=I(H)$ for some $H\in {\mathcal H}_E$ by \cite[Theorem 4.5]{APS}, so that $R/J=L(E)/I(H)\cong
L(F)$, where $F=E/H$. We must show that every nonzero right ideal $I$ of $L(F)$ contains an
infinite idempotent. First, apply \cite[Lemma 3.2]{APS} to get that $F$ satisfies Condition (K).

We will prove that every vertex $v$ in $F$ connects to a cycle in $F$. From \cite[Proposition
6.3]{AA3}, we know that Leavitt path algebras are semiprimitive, so there exists a (left)
primitive ideal $P$ of $L(F)$ such that $v\not\in P$. This ideal is, in particular, prime in
$L(F)$, and so corresponds by \cite[Proposition 5.6]{APS} to a maximal tail $M$ in $F$ in the
sense that $P=I_F(F^0\setminus M)$. Clearly then, $v\in M$. Moreover, $M$ is also a maximal tail
in $E$ as stated in \cite[Proof of Theorem 2.3]{HS}, so that $v$ connects to a cycle in $M$ (and
therefore in $F$) by the hypotheses of (vii).

Consider a nonzero element $x\in I$. Since $F$ has Condition (K), every cycle in $F$ has an exit
in $F$. An application of \cite[Proposition 6]{AA2} yields that there exist elements
$\alpha,\beta\in L(F)$ such that $\alpha x\beta=w\in F^0$. Because $w$ connects to a cycle, we can
find a (possibly trivial) path $\mu\in F^*$ such that $\mu^*w\mu=v$ where $v$ lies in a cycle.
Therefore $v=axb$ for certain $a,b\in L(E)$, where we can assume that $va=a$ and $bv=b$.

Write $f=xba$, which is an idempotent element of $I$. Moreover, $v=afxb$, and so $v\lesssim f$.
Since $F$ has Condition (K), we get that $|CSP(v)|\geq 2$. By Lemma \ref{vInfiniteIdempotent}, $v$
is an infinite idempotent, and therefore so is $f$.

\smallskip

(vii) $\Rightarrow$ (ii) is proved analogously.
\end{proof}

In \cite{AA2} and \cite{AGP}, the authors take as the definition of ``purely infinite'' for simple
rings the left-right symmetric condition ``every nonzero left ideal contains an infinite
idempotent''. From Theorem \ref{LeavittPi}, we see that our more general definition of purely
infinite ring agrees with that given in the simple case for Leavitt path algebras. Consequently,
we can immediately deduce the main result of \cite{AA2} as a corollary.

\begin{corollary} \label{pisLPAs} {\rm \cite[Theorem 11]{AA2}} Let $E$ be a row-finite graph. Then
$L(E)$ is purely infinite simple if and only if $E$ has the following properties:
\begin{enumerate}[{\rm (i)}]
\item The only hereditary and saturated subsets of $E^0$ are $\emptyset$ and $E^0$.
\item Every cycle in $E$ has an exit.
\item Every vertex in $E$ connects to a cycle.
\end{enumerate}
\end{corollary}

\begin{proof}
Suppose first that $L(E)$ is purely infinite simple. From the characterization of simple Leavitt
path algebras in \cite[Theorem 3.11]{AA1}, we obtain that (i) and (ii) hold. We next claim that
$E^0$ is a maximal tail in $E$. Trivially, the complement of $E^0$ is hereditary and saturated.
Now consider any two vertices $v,w\in E^0$. The set $H= \{x\in E^0 \mid v\ge x\}$ is clearly
hereditary, and so by (i), the saturated closure of $H$ must equal $E^0$. Consequently, there
exist hereditary subsets $H_1=H, H_2,\dots, H_n \subseteq E^0$ such that $w\in H_n$ and, for
$i=2,\dots,n$, we have
 \begin{enumerate}
 \item[] $H_i= H_{i-1}\cup \{w_i\}$ for some vertex $w_i$ which feeds into $H_{i-1}$ and only into
 $H_{i-1}$.
 \end{enumerate}
It follows that each $w_i$ feeds into $H$, and so there exists $y\in H$ such that $w\ge y$. By
definition of $H$, we also have $v\ge y$, proving that $E^0$ is indeed a maximal tail. Now Theorem
\ref{LeavittPi}(vi) implies that (iii) holds.

Conversely, suppose that (i), (ii) and (iii) hold. Use \cite[Theorem 3.11]{AA1} to see that $L(E)$
is simple. Since the complement of a maximal tail is hereditary and saturated, (i) implies that
the only possible nonempty maximal tail in $E$ is $E^0$. Hence, our current hypotheses imply
condition (vi) of Theorem \ref{LeavittPi}, and so condition (i) of that theorem says that every
nonzero right ideal of $L(E)$ contains an infinite idempotent. Therefore $L(E)$ is purely infinite
simple.
\end{proof}

\begin{remarks} \label{LEfacts} {\rm
We record a few useful facts about the elements of a Leavitt path algebra $L= L_K(E)$. Recall that
the term ``path'' is used to refer only to paths consisting of real edges.

(a) Distinct paths in $E$ are linearly independent elements of $L$ \cite[Lemma 1.1]{Sil}.

(b) If $p$ and $q$ are paths in $E$, then $q^*pq$ is either zero or a path of the same length as
$p$. For if $q^*pq\ne 0$, then either $p=qr$ for some path $r$, in which case $q^*pq= rq$, or else
$q=ps=st$ for some paths $s$, $t$, in which case $q^*pq= s^*q= t$.

(c) If $p_1,\dots,p_n$ are distinct paths, and $q$ is a path with $\deg(q) \le \deg(p_i)$ for all
$i$, then $q^*p_iq \ne q^*p_jq$ whenever $i\ne j$ and $q^*p_iq \ne 0$. To see this, arrange the
indexing so that $q^*p_iq \ne 0$ for $i=1,\dots,m$ and $q^*p_iq=0$ for $i=m+1,\dots,n$. For $i\le
m$, we must have $p_i= qr_i$ for a path $r_i$, and the $r_i$ must be distinct, so the paths
$q^*p_iq= r_iq$ are distinct.
 }\end{remarks}

\begin{lemma} \label{atensorv}
Let $E$ be a row-finite graph in which every cycle has an exit, and let $A$ be an s-unital
$K$-algebra. Given any nonzero element $x\in A\otimes_K L_K(E)$, there exist a nonzero element
$a\in A$ and a vertex $v\in E^0$ such that $a\otimes v \precsim x$.  \end{lemma}

\begin{proof} Set $L= L_K(E)$ and $R= A\otimes_K L$, write $x= \sum_j a_j \otimes b_j$ for some $a_j\in A$
and $b_j\in L$, and choose $u\in A$ such that $ua_j= a_ju= a_j$ for all $j$. There is at least one
vertex $v\in E^0$ such that $x(u\otimes v) \ne 0$, and we may replace $x$ by $x(u\otimes v)$, that
is, there is no loss of generality in assuming that $x= x(u\otimes v)$.

Let $\P$ denote the set of paths in $E$. This is a $K$-linearly independent subset of $L$ by
Remark \ref{LEfacts}(a), and so if $K\P$ denotes the $K$-span of $\P$ in $L$, then $A\otimes_K
K\P= \bigoplus_{p\in \P} A\otimes p$. Let us denote this subalgebra of $R$ by $A\P$.

We first claim that there is a path $\mu$ in $E$ such that $0\ne x\mu \in A\P$. This follows the
argument of \cite[Proposition 3.1]{AMMS}, which gives the claim in the case $A=K$, as observed in
\cite[proof of Proposition 2.2]{Sil}. We may of course assume that $x\notin A\P$. Write
$$x= \beta+ \sum_{i=1}^m \beta_i(u\otimes e_i^*)$$
 where  $\beta\in A\P$, the $e_i$ are distinct edges in $E^1$ with $s(e_i)=v$, and the $\beta_i$ are nonzero
elements of $R$. Assume also that the number $t$ of ghost edges needed to describe $x$ (including
 the $e_i^*$) is minimal for nonzero elements $x'\in R$ with $x'\precsim x$. Since $x\notin A\P$,
there must be at least one term in the displayed sum. Now $e_1^*= e_1^*v$ and so $v=s(e_1)$,
showing that $v$ is not a sink.

If $x(u\otimes e_j)\ne 0$ for some $j$, then $x(u\otimes e_j)$ is a nonzero element of $R$ with
$x(u\otimes e_j)\precsim x$ and $x(u\otimes e_j)= \beta (u\otimes e_j)+ \beta_j$. The number of
ghost edges needed to describe $x(u\otimes e_j)$ is the number needed to describe $\beta_j$, which
is less than the number $t$. This contradicts the minimality of $t$ unless $\beta_j=0$, in which
case $x(u\otimes e_j)= \beta (u\otimes e_j) \in A\P$ and our claim is proved.

Now suppose that $x(u\otimes e_i)= 0$ for all $i$. Then $\beta (u\otimes e_i)+ \beta_i=0$ for all
$i$, whence
$$\beta \bigl( u\otimes \bigl( v- \sum_{i=1}^m e_ie_i^* \bigr) \bigr)= \beta- \sum_{i=1}^m
\beta (u\otimes e_i)(u\otimes e_i^*) =x \ne 0.$$
 Consequently, $v- \sum_{i=1}^m e_ie_i^* \ne 0$, which means that $e_1,\dots,e_m$ are not the only
 edges emitted by $v$. If the others are $e_{m+1}, \dots,e_n$, then $v- \sum_{i=1}^m e_ie_i^* =
 \sum_{i=m+1} ^n e_ie_i^*$ and
$$x= x \bigl( u\otimes \sum_{i=m+1} ^n e_ie_i^*\bigr).$$
 It follows that $x(u\otimes e_j)\ne 0$ for some $j>m$. But $x(u\otimes e_j)= \beta (u\otimes e_j)
 \in A\P$, and again the claim is proved.

In view of the claim, we may now assume that $x\in A\P$, and so $x= \sum_{i=1}^n a_i\otimes p_i$
for some $a_i\in A$ and some distinct paths $p_i$ in $E$. We may also assume that the number of
terms, $n$, is minimal for such expressions of nonzero elements $x'\in A\P$ with $x'\precsim x$ in
$R$. In particular, all the $a_i\ne 0$. Arrange the indexing so that $\deg(p_1) \le \cdots\le
\deg(p_n)$.

Now $(u\otimes p_1^*)x = \sum_{i=1}^n a_i\otimes p_1^*p_i$ where each $p_1^*p_i$ is either zero or
a path in $E$. Moreover, $p_1^*p_1= v$ (recall that $x= x(u\otimes v)$), and those $p_1^*p_i$
which are nonzero are distinct. It follows that $(u\otimes p_1^*)x \ne 0$, and so we may replace
$x$ by $(u\otimes p_1^*)x$. Thus, there is no loss of generality in assuming that $p_1=v$. This
means that we are done if $n=1$, and so we may also assume that $n>1$. Note that for $i>1$, the
path $p_i\ne p_1=v$, so $\deg(p_i)>0$.

Next, note that $(u\otimes v)x (u\otimes v)= \sum_{i=1}^n a_i\otimes vp_iv$ where those $vp_iv$
which are nonzero are distinct. Since $vp_1v= v$, it follows that $(u\otimes v)x (u\otimes v)\ne
0$, and we replace $x$ by $(u\otimes v)x (u\otimes v)$. Thus, we may now assume that all the $p_i$
are closed paths based at $v$.

At this point, we have a closed path $p_2$ of positive length based at $v$, and so $p_2=p'_2p''_2$
where $p'_2$ is a closed simple path at $v$ and $p''_2$ is a closed path (possibly trivial) at
$v$. If $p'_2$ is a cycle, then it has an exit by hypothesis, while if it is not a cycle, it
automatically has an exit. Hence, $p'_2= qer$ for paths $q$ and $r$ and an edge $e$ such that
$s(e)$ emits an edge $f\ne e$. Then $f^*q^*p'_2= f^*er=0$, and so $f^*q^*p_2=0$. Consequently,
$$(u\otimes f^*q^*) x (u\otimes qf)= a_1\otimes r(f)+ \sum_{i=3}^n a_i\otimes f^*q^*p_iqf.$$
 Further, since $\deg(p_2) \le \deg(p_i)$ for $i>1$, those $f^*q^*p_iqf$ for $i>1$ which are
 nonzero are distinct paths of positive length. Hence, $(u\otimes f^*q^*) x (u\otimes qf) \ne 0$.
However, this contradicts the minimality of $n$.

Therefore we must have $n=1$, and the proof is complete.
\end{proof}

\begin{corollary} \label{idealsAtensorLE}
Let $A$ be an s-unital $K$-algebra, and let $E$ be a row-finite graph such that
\begin{enumerate}[{\rm (i)}]
\item The only hereditary and saturated subsets of $E^0$ are $\emptyset$ and $E^0$.
\item Every cycle in $E$ has an exit.
\end{enumerate}
Then every ideal of $A\otimes_K L_K(E)$ has the form $I\otimes_K L_K(E)$ for some ideal $I$ of
$A$.  \end{corollary}

\begin{remarknonum} {\rm This would follow from standard results when $E$ is finite, once we showed
that the center of $L_K(E)$ is $K$. The use of Lemma \ref{atensorv} saves that step, not to
mention extra techniques needed to investigate centers of corners when $E$ is infinite.
}\end{remarknonum}

\begin{proof} Set $L= L_K(E)$ and $R= A\otimes_K L$, and recall from \cite[Theorem 3.11]{AA1} that $L$
is a simple algebra. Given an ideal $J$ of $R$, define
$$I= \{ a\in A \mid a\otimes L \subseteq J \},$$
 and observe that $I$ is an ideal of $A$. Since $I\otimes_K L \subseteq J$, we may factor out
$I\otimes_K L$ and work in $(A/I)\otimes_K L$. Hence, there is no loss of generality in assuming
that $I=0$.

If there is a nonzero element $x\in J$, then by Lemma \ref{atensorv} there exist $a\in A$ and
$v\in E^0$ such that $0\ne a\otimes v \precsim x$. In particular, $a\otimes v \in J$. It now
follows that $a\otimes L= a\otimes LvL \subseteq J$ and $a\in I$, contradicting the assumption
that $I=0$. Therefore $J=0$, and the corollary is proved.  \end{proof}

\begin{theorem} \label{AtensorLEpinf}
Let $A$ be an s-unital $K$-algebra and $E$ a row-finite graph such that
\begin{enumerate}[{\rm (i)}]
\item Every nonzero right ideal in every quotient of $A$ contains a nonzero idempotent.
\item The only hereditary and saturated subsets of $E^0$ are $\emptyset$ and $E^0$.
\item Every cycle in $E$ has an exit.
\item Every vertex in $E$ connects to a cycle.
\end{enumerate}
Then the algebra $R= A\otimes_K L_K(E)$ is properly purely infinite. \end{theorem}

\begin{proof} Set $L= L_K(E)$. By Proposition \ref{IdealInfIdemImpliesPi}, it suffices to show
that every nonzero right
ideal in every nonzero quotient of $R$ contains an infinite idempotent. In view of Corollary
\ref{idealsAtensorLE}, every quotient of $R$ is isomorphic to an algebra of the form
$(A/I)\otimes_K L$ where $I$ is an ideal of $A$. Hence, after replacing $A$ by $A/I$, we just need
to show that every nonzero right ideal $J$ of $R$ contains an infinite idempotent.

Choose a nonzero element $x\in J$. By Lemma \ref{atensorv},  there exist $a\in A$ and $v\in E^0$
such that $0\ne a\otimes v \precsim x$. By hypothesis, there is a nonzero idempotent $e\in aA$,
whence $f=e\otimes v$ is a nonzero idempotent in $R$ such that $f\precsim a\otimes v\precsim x$.
Since $L$ is purely infinite simple by \cite[Theorem 11]{AA2}, the idempotent $v\in L$ is properly
infinite. Hence, Lemma \ref{elementtensorpropinf} implies that $f$ is properly infinite. But $f$
is equivalent to an idempotent in $J$ (recall the proof of Theorem \ref{AtensorL}), and therefore
$J$ contains a (properly) infinite idempotent. \end{proof}

\section*{Acknowledgements}
This work was partly carried out at the Centre de Recerca
Matem\`atica (Bellaterra), during the Programme ``Discrete and
Continous Methods in Ring Theory'' in 2006-07. We gratefully
acknowledge the support and hospitality extended to us.

\bibliographystyle{amsplain}

\end{document}